\documentclass[11pt, notitlepage]{article}
\usepackage{amssymb,amsmath,comment}
\catcode`\@=11 \@addtoreset{equation}{section}
\def\thesection{\arabic{section}}

\def\theequation{\thesection.\arabic{equation}}
\catcode`\@=12
\usepackage{colortbl}%
\usepackage[mathscr]{eucal}
\usepackage{epsf}
\usepackage{a4wide}
\def\R{\mathbb{R}}

\newcommand{\e}{\epsilon}

\newcommand{\Om} {\Omega}

\newcommand{\De} {\Delta}
\newcommand{\la} {\lambda}

\newcommand{\noi} {\noindent}

\newcommand{\mb} {\mathbb}
\newcommand{\mc} {\mathcal}

\usepackage[all]{xy}
\catcode`\@=11
\def\theequation{\@arabic{\c@section}.\@arabic{\c@equation}}
\catcode`\@=12

\newcommand{\QED}{\rule{2mm}{2mm}}

\newtheorem{Theorem}{Theorem}[section]
\newtheorem{Lemma}[Theorem]{Lemma}
\newtheorem{Proposition}[Theorem]{Proposition}
\newtheorem{Corollary}[Theorem]{Corollary}
\newtheorem{Remark}[Theorem]{Remark}
\newtheorem{Definition}[Theorem]{Definition}

\begin{document}

{\vspace{0.01in}}

\title
{ \sc Existence of three positive solutions for  a nonlocal singular dirichlet boundary problem}

\author{J. Giacomoni\footnote{Universit\'e de Pau et des Pays de l'Adour, CNRS, LMAP (UMR 5142) Bat. IPRA,
  Avenue de l'Universit\'e,
   64013 Pau cedex, France. e-mail:jacques.giacomoni@univ-pau.fr}, ~~ T. Mukherjee\footnote{Department of Mathematics, Indian Institute of Technology Delhi,
Hauz Khaz, New Delhi-110016, India.
 e-mail: tulimukh@gmail.com}~ and ~K. Sreenadh\footnote{Department of Mathematics, Indian Institute of Technology Delhi,
Hauz Khaz, New Delhi-110016, India.
 e-mail: sreenadh@gmail.com} }

\date{}

\maketitle

\begin{abstract}
In this article, we prove the existence of at least three positive solutions for the following nonlocal singular problem
\begin{equation*}
(P_\la)\left\{
\begin{split}
(-\De)^su &= \la\frac{f(u)}{u^q}, \; \; u>0 \;\; \text{in}\;\; \Om,\\
u &= 0\;\; \text{in}\;\; \mb R^n \setminus \Om
\end{split}
\right.
\end{equation*}
where $(-\De)^s$ denotes the fractional Laplace operator for $s\in (0,1)$, $n>2s$, $q \in (0,1)$, $\la>0$ and $\Om$ is smooth bounded domain in $\mb R^n$. Here $f :[0,\infty) \to [0,\infty)$ is a continuous nondecreasing map satisfying $\lim\limits_{u\to \infty}\frac{f(u)}{u^{q+1}}=0$. We show that under certain additional assumptions on $f$, $(P_\la)$ possesses at least three distinct solutions for a certain range of $\la$. We use the method of sub-supersolutions and a critical point theorem by Amann \cite{amann} to prove our results. Moreover, we prove a new existence result for a suitable infinite semipositone nonlocal problem which played a crucial role to obtain our main result and is of independent interest.
\medskip

\noi \textbf{Key words:} Fractional Laplacian, singular nonlinearity, infinite semipositone problem, sub-supersolutions, three positive solutions.

\medskip

\noi \textit{2010 Mathematics Subject Classification:} 35R11, 35R09, 35A15.

\end{abstract}

\section{Introduction}
In the present  paper, we consider the following nonlocal singular problem
\begin{equation*}
(P_\la)\left\{
\begin{split}
(-\De)^su &= \la\frac{f(u)}{u^q}, \; \; u>0 \;\; \text{in}\;\; \Om,\\
u &= 0\;\; \text{in}\;\; \mb R^n \setminus \Om
\end{split}
\right.
\end{equation*}
where $s\in (0,1)$, $q \in (0,1)$, $\la>0$ and $\Om$ is smooth bounded domain in $\mb R^n$.
We have the following assumptions on $f \in C^1([0,\infty))$:
\begin{enumerate}
\item[(f1)] $f(0)>0$,
\item[(f2)] $\lim\limits_{u \to \infty} \frac{f(u)}{u^{q+1}}=0$,
\item[(f3)] $u\mapsto f(u)$ is non decreasing in $\R^+$,
\item[(f4)] There exist a $0<\sigma_1<\sigma_2$ such that  $\frac{f(u)}{u^{q}}$ is non decreasing on $(\sigma_1,\sigma_2)$.
\end{enumerate}
Note that from (f1) we have $\lim\limits_{u \to 0} \frac{f(u)}{u^{q }}=\infty$.
{\begin{Remark}
It is easy to see that the function $f$ defined as $\displaystyle f(t)=e^{\frac{\alpha t}{\alpha+t}}$ for any $t\geq 0$ with $\alpha> 4q$ satisfy assumptions (f1)-(f4).
\end{Remark}
\noi The fractional Laplace operator $(-\De)^s$ is defined  as
$$ (-\De)^s u(x) = 2C^n_s\mathrm{P.V.}\int_{\mb R^n} \frac{u(x)-u(y)}{\vert x-y\vert^{n+2s}}\,\mathrm{d}y$$
{where $\mathrm{P.V.}$ denotes the Cauchy principal value and $C^n_s=\pi^{-\frac{n}{2}}2^{2s-1}s\frac{\Gamma(\frac{n+2s}{2})}{\Gamma(1-s)}$, $\Gamma$ being the Gamma function.} The fractional  Laplacian is the infinitesimal generator of {L\'evy} stable diffusion processes and arises in anomalous diffusion in plasma, population dynamics, geophysical fluid dynamics, flames propagation, chemical reactions in liquids and American options in finance, see \cite{da} for instance. {Fractional Sobolev spaces were introduced mainly in the framework of harmonic analysis in the middle part of last century and are the natural setting to study weak solutions to problems involving the fractional Laplacian. In this regard, the paper of Caffarelli and Silvestre \cite{CS} on the harmonic extension problem have subsequently motivated many works on equations and systems involving the fractional Laplacian $(-\De)^s$, $s \in(0,1)$.} We also refer \cite{radu2} to readers for a detailed study on variational methods for fractional elliptic problems and for additional references.

\noi In the local case , i.e. $s=1$, the study of elliptic singular problems starts mainly with the pioneering work of Crandal, Rabinowitz and Tartar \cite{crt}. This seminal work inspired a huge list of articles where authors have investigated many different issues (existence/nonexistence, uniqueness/multiplicity, regularity of solutions, etc.) about singular problems in the local and more recently in the nonlocal set up. We cite here some related works with no intent to furnish an exhaustive list. The multiplicity of solutions for singular problem with critical nonlinearity has been studied in \cite{haitoh,hcn1,hcn2} while the exponential critical nonlinearity has been dealt with in \cite{DJP}. Semilinear elliptic and singular problems with convection term was first studied in \cite{GhRa2} whereas \cite{DiMoOs} brought existence results to elliptic equations involving a singular absorption term. We refer the surveys \cite{GhRa} and \cite{HeMa} for further details on singular elliptic equations in the local setting.  In the nonlocal case, singular problem with critical nonlinearity has been studied in \cite{peral,TJS,TS}. Recently, Adimurthi, Giacomoni and Santra \cite{AJS} studied the following nonlocal singular problem:
\begin{equation}
(-\De)^su =\la (K(x)u^{-\delta}+f(u)),\; u>0\; \text{in}\; \Om,\;\;\; u = 0 \; \text{in} \; \mb R^n \setminus \Om
\end{equation}
where $\delta, \;\la>0$, $K :\Om \to \mb R^+$ is a H\"{o}lder continuous function in $\Om$ behaves like dist$(x,\partial\Om)^{-\beta}$, $\beta \in [0,2s)$ and $f$ is a positive real valued $C^2$ function. They established existence, regularity and bifurcation results using the framework of weighted spaces.

\noi Recently, \cite{DI} has {established} the existence of three non-zero solutions for  a Dirichlet type boundary value problem involving the fractional Laplacian. But the study of three solutions for singular nonlocal problems was completely open till now. Our work brings new results in this regard. We use the method of sub and supersolution combined with a fix point theorem due to Amann to achieve the objective. For the construction of the barrier functions, we have taken some ideas from \cite{KLS}.

\noi The salient feature of this work is the presence of the singular term $u^{-q}$ which is a primary hindrance in making the operator monotone. We slightly transform the problem to a new one and show that the operator associated with it becomes monotone increasing and compact. This idea had been formerly used in \cite{Dhanya-JMAA} in the local case. But here itself we remark that their approach can not be directly applied to the problem $(P_\la)$ due to the presence of the nonlocal operator '$(-\De)^s $' instead of $\De$. Most substantially, Theorem $3.6$ of \cite{Dhanya-JMAA} can not be adapted here due to the lack of an explicit form of $(-\De)^s \delta^s(x)$  where $\delta(x)=\text{dist}(x,\partial \Om)$ denotes the distance function up to the boundary. To overcome this difficulty we construct a subsolution $v$ satisfying
\[(-\De)^s v + \frac{c v}{\delta^q(x)} \leq 0 \; \text{in}\; \Om_r \]
where $c>0$ is constant and $\Om_r \subset \Om$. To obtain this, we separately study, by bifurcation arguments, a nonlocal infinite semipositone problem $(I_\theta)$ in section $6$. This leads naturally to a solution of the required problem. In the local setting, we refer to readers \cite{shiv1,shiv2,shiv3} concerning infinite semipositone problems. {But we indicate} that 'nonlocal' infinite semipositone problem has not been studied in the past. So our results are completely new in this regard.  The main result of our paper is accomplished by using a well known critical point theorem by Amann \cite{amann}.  Last but not the least, we additionally prove the uniqueness of solutions  to $(P_\la)$ when $\la$  becomes sufficiently large under appropriate condition of '$f$'. This result is motivated by the paper \cite{CES} of Castro, Eunkyung and Shivaji. Nevertheless, we point out that their approach can not be exactly applied here due to the presence of the nonlocal operator '$(-\De)^s$'. We still succeed to obtain the result by an appropriate application of the Hardy's inequality for fractional Laplacian (refer section 5).
 Now we state the main results of our paper as follows.

\begin{Theorem}\label{3sol}
There exists constants $\la_1,\la_2>0$ such that  if $\la \in [\la_1,\la_2]$ then the problem $(P_\la)$ has at least three solutions in $C_{\phi_{1,s}}^+(\Om)$.
\end{Theorem}
\begin{Remark}
Theorem \ref{3sol} still holds if (f3) is replaced by the weakened assumption (f3'): There exists $k>0$ such that $u\mapsto f(u)+ku$ is non decreasing in $\R^+$.
\end{Remark}
\begin{Theorem}\label{unique-sol}
There exist a $\lambda^*>0$ such that $(P_\la)$ has a unique solution when $\la > \la^*$.
\end{Theorem}
\begin{Remark}
Since $f(0)>0$, it is not difficult to show that for $\lambda>0$ small enough, there exists a unique solution with small norm. Then Theorem~\ref{3sol} and \ref{unique-sol} entail that the bifurcation curve of solutions to $(P_\la)$ emanating from $(0,0)$ is S-shaped.
\end{Remark}
The outline of this paper is as follows: In the second Section we give some useful preliminaries about the main equation in $(P_\lambda)$. In the third Section, we construct the sub and supersolutions used to apply the fixed point theorem of Amann. In Section 4, we prove the main result of our paper: Theorem~\ref{3sol}. In Section 5, we prove our main uniqueness result: Theorem~\ref{unique-sol} and finally in Section 6, we investigate the fractional and singular semipositone problem $(I_\theta)$ used crucially for proving the strong increasingness of the operator $T$ in Section 4.
\section{Preliminaries}
We start with defining the function spaces. Given any $\phi \in C_0(\overline \Om)$ such that $\phi >0$ in $\Om$ we define
\[C_{\phi}(\Om):= \{u \in C_0(\Om)| \;\exists \;c\geq0 \;\text{such that}\; |u(x)|\leq c\phi(x), \; \forall x \in \Om \}\]
with the usual norm $\displaystyle \left\|\frac{u}{\phi}\right\|_{L^\infty(\Om)}$ and the associated positive cone. We define the following open convex subset of $C_{\phi}(\Om)$ as
\[C_{\phi}^+(\Om):= \left\{ u \in C_{\phi}(\Om)|\; \inf_{x\in \Om}\frac{u(x)}{\phi(x)}>0 \right\}.\]

In particular, $C_{\phi}^+$ contains all those functions $u \in C_0(\Om)$ with $k_1\phi \leq u\leq k_2 \phi$ in $\Om$ for some $k_1,k_2>0$. We consider the following fractional Sobolev space
\[\tilde H^s(\Om):= \{u \in H^s(\mb R^n): \; u = 0 \; \text{in}\; \mb R^n \setminus \Om\}\]
equipped with the norm
\[\|u\|= \left(\int_Q \frac{|u(x)-u(y)|^2}{|x-y|^{n+2s}}~\mathrm{d}x\mathrm{d}y\right)^{\frac12},\; \text{where}\; Q= \mb R^{2n}\setminus (\mc C\Om\times\mc C\Om).\]
\begin{Definition}
We say that $u \in \tilde H^s(\Om)$ is a weak solution to $(P_\la)$ if $\inf\limits_K u>0$ for every compact subset $K \subset \Om$ and for any $\varphi \in \tilde H^s(\Om)$,
\begin{equation}\label{weaksol}
\int_{\Om}(-\De)^su\varphi = C^n_s\int_Q\frac{(u(x)-u(y))(\varphi(x)-\varphi(y))}{|x-y|^{n+2s}}\mathrm{d}x\mathrm{d}y= \la\int_{\Om} \frac{f(u)}{u^q}\varphi\mathrm{d}x.
\end{equation}
\end{Definition}
\begin{Definition}
By a subsolution of the problem $(P_\la)$, we mean a function $v \in \tilde H^s(\Om)$ which satisfies (weakly)
\begin{equation}\label{subsol}
(-\De)^s v \leq \la \frac{f(v)}{v^q},\; v>0 \; \text{in}\; \Om,\;\;v =0 \;\text{in}\; \mb R^n \setminus \Om.
\end{equation}
Whereas if the reverse inequality holds in \eqref{subsol}, we call $v$ to be a supersolution of $(P_{\la})$. Also we call them strict sub and supersolution if the inequality in \eqref{subsol} is strict.
\end{Definition}
We define the distance function as $\delta(x):= \text{dist}(x,\partial \Om)$, $x \in \Om$. Let $\phi_{1,s}$ denotes the first positive eigenfunction of $(-\De)^s$ in $\tilde H^s(\Om)$ corresponding to its principal eigenvalue $\la_{1,s}$ such that $\|\phi_{1,s}\|_{L^\infty(\Om)}=1$. We recall that $\phi_{1,,s}\in C^s(\mb R^n)$ and also $\phi_{1,s}\in C_{\delta^s}^+(\Om)$ (see for instance Proposition $1.1$ and Theorem $1.2$ of \cite{RoSe}).

\section{Sub and Supersolutions of $(P_\la)$}
In this section we show the existence of two pairs of sub-supersolutions $(\zeta_1,\vartheta_1)$ and $(\zeta_2,\vartheta_2)$ such that $\zeta_1\leq \zeta_2\leq \vartheta_1$, $\zeta_1\leq \vartheta_2 \leq \vartheta_1$ and $\zeta_2\nleqslant \vartheta_2$. Moreover it holds that $\zeta_2,\vartheta_2$ are strict sub and supersolutions of $(P_\la)$. Let $w$ denotes the unique solution of the problem
\begin{equation}\label{psp}
(-\De)^sw= \frac{1}{w^q},\; w>0\; \text{in}\; \Om,\; \; w=0\;\text{in}\; \mb R^n\setminus \Om.
\end{equation}
Then from the proof of Theorem $1.1$ of \cite{AJS}, we know that $w \in \tilde H^s(\Om)\cap C_{\phi_{1,s}}^+(\Om)$ and $w \in C^s(\mb R^n)$. We construct our supersolution $\vartheta_1$ first. Since (f2) holds, we get $\lim\limits_{u \to \infty}\frac{ f(u)}{u^{q+1}}=0$. This implies that if we choose a constant $M_\la \gg 1$ sufficiently large such that
\[\frac{ f(M_\la\|w\|_{L^\infty(\Om)})}{(M_\la\|w\|_{L^\infty{(\Om)}})^{q+1}}\leq \frac{1}{\la \|w\|_{L^\infty(\Om)}^{q+1}}\;\;\text{i.e}\;\; M_\la^{q+1} \geq \la f(M_\la\|w\|_{L^\infty(\Om)})
\]
then $\vartheta_1= M_\la w \in \tilde H^s(\Om)\cap C_{\phi_{1,s}}^+(\Om)$ forms a supersolution of $(P_\la)$. Indeed using non decreasing nature of $f$ we get
\begin{align*}
(-\De)^s \vartheta_1= \frac{M_\la^{q+1}}{(M_\la w)^q}\geq \la\frac{ f(M_\la\|w\|_{L^\infty(\Om)})}{(M_\la w)^q} \geq \la \frac{  f(M_\la w)}{(M_\la w)^q} =\la \frac{  f(\vartheta_1)}{(\vartheta_1)^q}.
\end{align*}
Now since $\lim\limits_{u \to 0} \frac{f(u)}{u^{q}}=\infty$, we can choose $m_\la>0$ sufficiently small so that
\[\la_{1,s}m_\la \phi_{1,s} \leq \la \frac{f(m_\la\phi_{1,s})}{(m_\la \phi_{1,s})^q},\; \; \text{for each}\; \la >0. \]
Now we define $\zeta_1 = m_\la \phi_{1,s} \in \tilde H^s(\Om)\cap C_{\phi_{1,s}}^+(\Om)$ and it is easy to see that
\[(-\De)^s \zeta_1= m_\la \la_{1,s} \phi_{1,s}\leq \la\frac{f(m_\la \phi_{1,s})}{(m_\la \phi_{1,s})^q} = \la \frac{f(\zeta_1)}{\zeta_1^q} .  \]
Therefore $\zeta_1$ is a subsolution of $(P_{\la})$. It is not hard to see that we always choose $m_\la$ small enough so that $\zeta_1 \leq \vartheta_1$. This completes our construction of first pair of sub-supersolution.

\noi Our next step is to construct the second pair of sub-supersolution of $(P_\la)$. We first construct our positive supersolution $\vartheta_2$ such that $\|\vartheta_2\|_{L^\infty(\Om)}= \sigma_1$(see (f4)). Let us define $\vartheta_2 = \frac{\sigma_1 w}{\|w\|_{L^\infty(\Om)}}\in \tilde H^s(\Om)\cap C_{\phi_{1,s}}^+(\Om)$ and assume that
\[0 < \la \leq \frac{\sigma_1^{q+1}}{f(\sigma_1)\|w\|_{L^\infty(\Om)}^{q+1}}.\]
Then using the non decreasing nature of $f$ we find that it satisfies
\begin{align*}
(-\De)^s\vartheta_2 = \frac{\sigma_1}{\|w\|_{L^\infty(\Om)}w^q}\geq \la \frac{f(\sigma_1)\|w\|_{L^\infty(\Om)}^q}{(\sigma_1 w)^q}\geq \la \frac{f\left(\frac{\sigma_1 w}{\|w\|_{L^\infty(\Om)}}\right)}{\left(\frac{\sigma_1 w}{\|w\|_{L^\infty(\Om)}}\right)^q}= \la \frac{f(\vartheta_2)}{\vartheta_2^q}.
\end{align*}
Now we construct our second positive supersolution of $(P_{\la})$ which is one of the crucial part of our paper. For this, we let $\sigma \in (0,\sigma_1]$ be such that $ f^*(\sigma)= \min\limits_{0<x\leq \sigma} \displaystyle\frac{f(x)}{x^q}$ and also define $h \in C([0,\infty))$ such that
\begin{equation*}
h(u)=\left\{
\begin{split}
{f^*(\sigma)}&,\; \text{if}\; u\leq \sigma\\
\frac{f(u)}{u^q}&,\; \text{if}\; u\geq\sigma_1
\end{split}
\right.
\end{equation*}
so that $h$ is a non decreasing function on $(0,\sigma_1]$ and $h(u) \leq \frac{f(u)}{u^q}$ for all $u \geq 0$. With this definition of $h$, we consider a non singular problem
\begin{equation*}
(Q_\la)\left\{
\begin{split}
(-\De)^su=\la h(u)\; \text{in}\; \Om,\;\; u=0 \; \text{in}\; \mb R^n \setminus \Om.
\end{split}
\right.
\end{equation*}
Let $G_s(x,y)$ denote the Green function associated to $(-\De)^s$ with homogeneous Dirichlet boundary condition in $\Om$.  Then we have
\begin{equation*}
u(x)=\left\{
\begin{split}
\la \int_{\Om} G_s(x,y)h(y) ~\mathrm{d}y &,\; \text{if}\; x \in \Om,\\
0&,\; \text{if}\; x \in \mb R^n \setminus \Om.
\end{split}
\right.
\end{equation*}
Let $B_{\hat R}(0)$ denotes the ball (centered at $0$ where w.l.o.g. we assume that $0 \in \Om$) of largest radius $\hat R$ that is inscribed in $\Om$ and also let $R < \hat R$.
Suppose $K: L^2(\Om) \to L^2(\Om)$ be the linear map defined as
\[K(g)(x) = \int_{\Om}G_s(x,y)g(y)~\mathrm{d}y.\]
Let $\chi_{R}: \Om \to \mb R$ be the characteristic function defined as
\begin{equation*}
\chi_{R}(x)=\left\{
\begin{split}
1 &,\; \text{if}\; x \in B_{R}(0),\\
0&,\; \text{if}\; x \in \Om \setminus B_R(0).
\end{split}
\right.
\end{equation*}
Then from Theorem $1.1$ of \cite{Chen-song}, for each $(x,y) \in \Om \times \Om$ we have that
\begin{equation}\label{esti-on-G}
0\leq G_s(x,y) \leq C \min \left\{\frac{\delta^s(x)\delta^s(y)}{|x-y|^n}, \frac{\delta^s(x)}{|x-y|^{n-s}}\right\}.
\end{equation}
Therefore there exists a constant $C_1>0$ depending in $R$ such that
\[K(\chi_{R})(x)= \int_{\Om} G_s(x,y)~\mathrm{d}y\leq C \delta^s(x)\int_{B_R(0)}\frac{\mathrm{d}y}{|x-y|^{n-s}}\leq  C_1 \]
uniformly in $x \in \Om$. Let $M_2 = \left(\min\limits_{x \in \Om} K(\chi_{R})(x)\right)^{-1}>0$ and $a \in (\sigma_1,\sigma_2]$. Also we define $v= a \chi_{{R}}$ in $\Om$. Then if $v_1$ denotes a solution to the problem
\begin{equation}\label{def-v1}
(-\De)^s v_1 = h(v)\; \text{in}\; \Om, \;\; v_1=0 \; \text{in}\; \mb R^n\setminus \Om,
\end{equation}
then for each $x \in \Om$ we get
\[v_1(x) = \la \int_{\Om} G_s(x,y)h(a\chi_{R})(y)~\mathrm{d}y \leq \la h(a) \int_{\Om} G_s(x,y)~\mathrm{d}y  \]
where we used the fact that $h$ is non decreasing and $\chi_{R}\leq 1$ in $\Om$. Therefore $v_1 \leq \sigma_2$ in $B_R(0)$ if
\[\la \leq \frac{M_3 \sigma_2}{h(a)}, \; \text{where}\; M_3 = \left(\max\limits_{x \in \Om} \int_{\Om} G_s(x,y)~\mathrm{d}y  \right)^{-1}<+\infty.\]
\textbf{Claim:} $v_1\geq v$ in $\Om$ for certain range of $\la$.\\
Let $x \in \Om\setminus B_{R}(0)$ then since $G_s(x,y)\geq 0$ for all $x,y \in \Om$ and $h(u)>0$ for all $u \geq 0$ we get $v(x)=0 \leq v_1(x)$. Now let $ x\in B_{R}(0)$ then
\begin{align*}
v_1(x) &= \la \int_{\Om} G_s(x,y)h(a\chi_{R})(y)~\mathrm{d}y \\
&= \la\left( \int_{B_{R}(0)} G_s(x,y)h(a)~\mathrm{d}y+ \int_{\Om \setminus B_{R}(0)} G_s(x,y)h(0)~\mathrm{d}y\right)\\
& \geq \la h(a)\int_{B_{R}(0)} G_s(x,y)~\mathrm{d}y, \; \text{since}\; h(0)= \frac{f(\sigma)}{\sigma^q}>0 \; \text{and}\; G_s(x,y)\geq 0\\
& \geq \la \frac{h(a)}{M_2}.
\end{align*}
So if we assume that $\la \geq \displaystyle \frac{M_2 a}{h(a)}$ then by definition of $v$ we get
\begin{equation}\label{sub-sol-pos}
v_1(x) \geq a \geq v(x), \; \text{forall}\; x \in \Om.
\end{equation}
Hence we finally get that $0\leq v\leq v_1 \leq \sigma_2$ in $\Om$ if
\[\frac{M_2a}{h(a)} \leq \la \leq \frac{M_3 \sigma_2}{h(a)}.\]
Since we assumed $h$ to be non decreasing in $(0,\sigma_2)$ (because of (f4)) we get $h(v(x))\leq h(v_1(x))$, for $x \in \Om$. So $v_1$ weakly satisfies the problem
\[(-\De)^sv_1 \leq \la h(v_1)\;\text{in}\; \Om, \;\; v_1=0\;\text{in}\; \mb R^n \setminus \Om\]
Since $h(v) \in L^\infty(\Om)$, using Theorem $1.2$ of \cite{RoSe} we get that $v_1\in C^s(\mb R^n)$. This gives us that $v_1$ forms a subsolution of $(Q_\la)$. Moreover \eqref{def-v1} and the strong maximum principle implies that $v_1>0$ in $\Om$. Therefore, using the fact that $h(u) \leq \displaystyle \frac{f(u)}{u^q}$ for all $u \geq 0$ implies that $\zeta_2 =v_1$ forms a positive subsolution of $(P_{\la})$.

\noi Therefore we got that if
\[\la_1:= \frac{M_2a}{h(a)} \leq \la \leq \min\left\{\frac{\sigma_1^{q+1}}{f(\sigma_1)\|w\|^{q+1}_{L^\infty(\Om)}} , \frac{M_3\sigma_2}{h(a)} \right\}=:\la_2\]
then  we obtain a positive subsolution $\zeta_2$ and a positive supersolution $\vartheta_2$ of $(P_\la)$ such that $\zeta_2 \nleq \vartheta_2$. Indeed, $\|\vartheta_2\|_{L^\infty(\Om)}=\sigma_1$ and $\|\zeta_2\|_{L^\infty(\Om)} \geq a >\sigma_1$.

\section{Proof of main result}
In this section we prove our main result after establishing some necessary results. We begin by noticing that our problem $(P_\la)$ can be rewritten as
\begin{equation*}
(\tilde P_\la)\left\{
\begin{split}
(-\De)^su-\la \frac{f(0)}{u^q} = \tilde f(u),\; u>0\; \text{in}\;\Om,\;\; u = 0 \; \text{in}\; \mb R^n \setminus \Om
\end{split}
\right.
\end{equation*}
where $\tilde f(u)=\la\left( \displaystyle\frac{f(u)-f(0)}{u^q}\right)$. We fix that $\la \in [\la_1,\la_2]$. Since $f\in C^1([0,\infty))$ , by Mean value theorem we get $\tilde f(u)=\la f^\prime(v)u^{1-q}$ for some $v \in (0,u)$. Also this implies $\tilde f(0)=0$ because $\lim\limits_{t\to 0}|f^\prime (t)|<\infty$ and $q\in (0,1)$. Therefore $\tilde f$ can be considered as a continuous function on $[0,\infty)$ such that $\tilde f(0)=0$. We assume also the following-\\
(h1) There exists a constant $\tilde k>0$ such that $\tilde f(t)+\tilde k t$ is increasing in $[0,\infty)$.\\
\begin{Definition}
We say that $z \in \tilde H^s(\Om)$ is a weak solution of $(\tilde P_\la)$  if $\inf\limits_K z>0$ for every compact subset $K \subset \Om$ and for any $\varphi \in C_c^\infty(\Om)$,
\begin{equation}\label{weaksol1}
 C^n_s\int_Q\frac{(z(x)-z(y))(\varphi(x)-\varphi(y))}{|x-y|^{n+2s}}~\mathrm{d}x\mathrm{d}y-\la f(0)\int_{\Om}\frac{\varphi}{z^q}~\mathrm{d}x = \int_{\Om} \tilde f(z)\varphi ~\mathrm{d}x.
\end{equation}
\end{Definition}
We remark that for any $\varphi \in \tilde H^s(\Om)$ and $z(x) \geq k_1 \delta^s(x)$ in $\Om$, Hardy's inequality gives that
\[\left|\int_\Om \frac{\varphi}{z^q}~\mathrm{d}x\right| \leq k_1 \int_\Om \frac{|\varphi(x)|}{\delta^s(x)} \delta^{s(1-q)}(x) \mathrm{d}x\leq k_1 \left(\int_\Om \frac{|\varphi(x)|^2}{\delta^{2s}(x)}\right)^{\frac12} \left( \int_\Om \delta^{2s(1-q)}(x)\mathrm{d}x\right)^{\frac12}\leq C\|\varphi\|\]
where $k_1,C>0$ are constants. So now following the arguments of Lemma $3.2$ of \cite{TJS} we can prove that a weak solution $z \in \tilde H^s(\Om)$ of $(\tilde P_\la)$ satisfies \eqref{weaksol1} for every $\varphi \in \tilde H^s(\Om)$ if $z \geq k_1 \delta^s(x)$ in $\Om$.

We extend the function $f$ and $\tilde f$ naturally as $f(t)=f(0)$ and $\tilde f(t)= \tilde f(0)$ for $t \leq 0$. Because of the assumption in (h1), w.l.o.g.  we can assume that $\tilde f$ is increasing in $\mb R^+$.
Now we define a map $T: C_0(\overline \Om)\to C_{\phi_{1,s}}(\Om)$ as $T(u)=z$ if and only if $z$ is a weak solution of
\begin{equation*}
(S_\la)\left\{
\begin{split}
(-\De)^sz-\la \frac{f(0)}{z^q} = \tilde f(u),\; z>0\; \text{in}\;\Om,\;\; z = 0 \; \text{in}\; \mb R^n \setminus \Om.
\end{split}
\right.
\end{equation*}
By saying that $z\in \tilde H^s(\Om)$ is a weak solution of $(S_\la)$, we mean that it satisfies
\begin{equation}\label{weaksol2}
 C^n_s\int_Q\frac{(z(x)-z(y))(\varphi(x)-\varphi(y))}{|x-y|^{n+2s}}\mathrm{d}x\mathrm{d}y-\la f(0)\int_{\Om}\frac{\varphi}{z^q}\mathrm{d}x = \int_{\Om} \tilde f(u)\varphi\mathrm{d}x.
\end{equation}
for all $\varphi \in C_c^\infty(\Om)$. But repeating the arguments as above for $(\tilde P_\la)$ we can show that $z \in \tilde H^s(\Om)$ of $(S_\la)$ satisfies \eqref{weaksol2} for every $\varphi \in \tilde H^s(\Om)$ if $z \geq k_1 \delta^s(x)$ in $\Om$.

\begin{Proposition}\label{fixed-pt}
A function $z\in \tilde H^s(\Om)\cap C_{\phi_{1,s}}(\Om)$ is a weak solution of $(S_\la)$ if and only if $z$ is a fixed point of $T$.
\end{Proposition}
\begin{proof}
Suppose $z\in \tilde H^s(\Om)\cap C_{\phi_{1,s}}(\Om)$ is a weak solution of $(S_\la)$ then it is clear that $z$ forms a fixed point of the map $T$. Conversely assume $T(z)=z$ for some $z\in C_0(\overline \Om)$. Then it satisfies \eqref{weaksol2} but it remains to show that $z \in C_{\phi_{1,s}}^+(\Om)$. Since $z>0$ in $\Om$ and $\tilde f(z)$ is locally H$\ddot{\text{o}}$lder continuous in $\Om$, we can follow the arguments of Theorem $1.2$ of \cite{AJS} to obtain that $z \in C_{\phi_{1,s}}^+(\Om)$.
 \hfill{\QED}
\end{proof}

\begin{Proposition}\label{well-defined}
The map $T$ is well defined from $C_0(\overline \Om)$ to $C_{\phi_{1,s}}^+(\Om)$.
\end{Proposition}
\begin{proof}
Let $u \in C_0(\overline \Om)$ and define $v= \tilde f \circ u$ then $v \in C_0(\overline \Om)$ and $v \geq 0$ in $\Om$. To show that $T$ is a well defined map we need to show that $(S_\la)$ has a unique solution corresponding to the above $u$. We introduce the following approximated problem for $\e>0$
\begin{equation*}
(S^\e_\la)\left\{
\begin{split}
(-\De)^sz-\la \frac{f(0)}{(z+\e)^q} = \tilde f(u),\; z>0\; \text{in}\;\Om,\;\; z = 0 \; \text{in}\; \mb R^n \setminus \Om.
\end{split}
\right.
\end{equation*}
Then $(S_\la^\e)$ has a unique solution in $\tilde H^s(\Om)$. Indeed, let $\tilde H^s(\Om)^+$ denote the positive cone of $\tilde H^s(\Om)$ and define the energy functional $E_\e: \tilde H^s(\Om)^+ \rightarrow  \mb R$ as
\[ E_\e(z) = \frac{\|z\|^2}{2} - \la \frac{f(0)}{1-q}\int_{\Om}(z+\e)^{1-q}~\mathrm{d}x- \int_{\Om}\tilde f(u)z~\mathrm{d}x  \]
where $z \in \tilde H^s(\Om)^+$. Then $E_\e$ is weakly lower semi-continuous, strictly convex and coercive in $\tilde H^s(\Om)^+$.  Therefore, $E_\e$ admits a unique minimizer, say $z_\e \not\equiv 0$ in $\tilde H^s(\Om)^+$. Since for small $t>0$ the term $\displaystyle t^{1-q}\int_{\Om}(z+\e)^{1-q}~\mathrm{d}x $  dominates so $E_\e(tz)$ can be made small enough and we get $\inf\limits_{\tilde H^s(\Om)^+} E_\e<0$. We choose $m>0$ (independent of $\e$) sufficiently small such that
\[m\la_{1,s}\phi_{1,s} \leq v+ \la \frac{f(0)}{(m\phi_{1,s}+1)^q}\; \text{in}\; \Om.\]
Then we get that
\begin{equation}\label{wd1}
(-\De)^s(m\phi_{1,s}) = m\la_{1,s}\phi_{1,s} \leq v+ \la \frac{f(0)}{(m\phi_{1,s}+1)^q}\leq v+ \la \frac{f(0)}{(m\phi_{1,s}+\e)^q}\; \text{in}\; \Om.
\end{equation}
\textbf{Claim (1):} $m\phi_{1,s}\leq z_\e$, for each $\e>0$.\\
We define $\tilde z_\e := (m\phi_{1,s}-z_\e)^+$ and assume that meas(Supp$\tilde z_\e$) is non zero. Then $\eta : [0,1]\to \mb R$ defined by $\eta(t) = E_\e(z_\e+t\tilde z_\e)$ is a convex function since $E_\e|_{\tilde H^s(\Om)^+}$ is convex. Also
 \begin{align*}
\eta^\prime(1) =   C_s^n\int_Q\frac{( (z_\e +\tilde z_{\e})(x)-(z_\e + \tilde z_{\e} )(y))(\tilde z_{\e}(x)-\tilde z_{\e}(y))}{|x-y|^{n+2s}}~dxdy -\la\int_{\Om}\frac{f(0)\tilde z_{\e}}{( z_\e +\tilde z_{\e} +\e)^q}-\int_{\Om}\tilde f(u)\tilde  z_{\e}
\end{align*}
in $(0,1]$. The fact that $z_\e$ is a minimizer of $E_\e$ gives that $\lim\limits_{t\to 0^+}\eta^\prime(t)\geq 0$ and $0 \leq \eta^\prime (0) \leq \eta^\prime(1)$.  Let us recall the following inequality for any $\psi$ being a convex Lipschitz function:
 \[(-\De)^s\psi(u) \leq \psi^\prime (u)(-\De)^su.\]
 Therefore using this with $\psi(x)=\max \left\{x,0\right\}$ and \eqref{wd1}, we get
$\eta^\prime(1)\leq \langle E_\e^\prime(m\phi_{1,s}), \tilde z_\e \rangle <0$ which is a contradiction. Hence $\text{supp}(\tilde z_\e)$ must have measure zero which establishes the claim (1). Thus $E_\e$ is G\^{a}teaux differentiable at $z_\e$ and $z_\e$ satisfies $(S^\e_\la)$ weakly. Since
\[\tilde f(u)+ \la \frac{f(0)}{(z_\e+\e)^q}\in L^\infty(\Om),\; \text{for each}\; \e>0,\]
from Proposition $1.1$ and Theorem $1.2$ of \cite{RoSe} and claim (1) we get that $z_\e \in C^s(\mb R^n) \cap C_{\phi_{1,s}}^+(\Om)$ for each $\e>0$. Thus following the arguments in proof of Theorem $1.1$(p.7) of \cite{AJS} we can show that $\{z_\e\}_{\e>0}$ is a monotone increasing sequence as $\e\downarrow 0^+$ that is for $0<\e<\e^\prime$, it must be $z_{\e^\prime}< z_{\e}$ in $\Om$. Thus we infer that $z= \lim\limits_{\e \downarrow 0^+}z_\e \geq m\phi_{1,s}$. From $z_\e$ satisfying $(S^\e_\la)$ we obtain
\begin{equation}\label{wd2}
\|z_\e\|^2 = \la \int_{\Om}\frac{f(0)z_\e}{(z_\e+\e)^q}~\mathrm{d}x+ \int_{\Om}\tilde f(u)z_\e ~\mathrm{d}x.
\end{equation}
We recall the function $w\in \tilde H^s(\Om)\cap C_{\phi_{1,s}}^+(\Om)$ satisfying \eqref{psp}. Let $\overline z= Mw$ for $M\gg 1$ (independent of $\e$) sufficiently large so that
\[M \left( \frac{1}{w^q}- \frac{\la f(0)}{(Mw)^q}\right) > \tilde f(u)\; \text{in}\; \Om.\]
Then $\overline z$ satisfies
\[(-\De)^s\overline z -\la \frac{f(0)}{(\overline z+\e)^q} =  \frac{M}{w^q}- \frac{\la f(0)}{(M w+\e)^q} > M \left( \frac{1}{w^q}- \frac{\la f(0)}{(Mw)^q}\right) > \tilde f(u) \; \text{in}\; \Om.\]
Now we prove that $z_\e\leq \overline z$ by using a comparison argument, which we will refer as comparison principle in future. We know that $h= (z_\e - \overline z) \in \tilde H^s(\Om)$ satisfies the equation
\begin{equation}\label{wd3}
(\De)^s(z_\e-\overline z) \leq \la f(0)\left( \frac{1}{(z_\e+\e)^q}-\frac{1}{(\overline z+\e)^q}  \right)\; \text{in}\; \Om.
\end{equation}
If we denote $h^+=\max\{h,0\}$ and $h^-=-\min\{h,0\}$ then $h=h^+-h^-$. Let $\Om_h^+= \{x\in \Om:\; z_e>\overline z\}$ and $\Om^-_h= \Om \setminus \Om_h^+$ then testing \eqref{wd3} with $h^+$ gives
\begin{equation}\label{wd4}
C^n_s\int_Q \frac{(h(x)-h(y))(h^+(x)-h^+(y))}{|x-y|^{n+2s}}~\mathrm{d}x\mathrm{d}y \leq \la f(0) \int_{\Om_h^+}\left( \frac{1}{(z_\e+\e)^q}-\frac{1}{(\overline z+\e)^q}  \right)h^+~ \mathrm{d}x.
\end{equation}
It is easy to see that $(h(x)-h(y))(h^+(x)-h^+(y)) = h(x)h^+(x)\geq 0$ on $\Om \times \mc C\Om$ and $(h(x)-h(y))(h^+(x)-h^+(y))\geq 0$ on $\Om_h^+ \times \Om_h^-$. This gives
\[0< \int_{\Om_h^+}\int_{\Om_h^+} \frac{(h(x)-h(y))(h^+(x)-h^+(y))}{|x-y|^{n+2s}}~\mathrm{d}x\mathrm{d}y \leq \int_Q \frac{(h(x)-h(y))(h^+(x)-h^+(y))}{|x-y|^{n+2s}}~\mathrm{d}x\mathrm{d}y.\]
Therefore from \eqref{wd4} we obtain
\[ 0<C^n_s\int_{\Om_h^+}\int_{\Om_h^+} \frac{(h(x)-h(y))(h^+(x)-h^+(y))}{|x-y|^{n+2s}}~\mathrm{d}x\mathrm{d}y \leq \la f(0) \int_{\Om_h^+}\left( \frac{1}{(z_\e+\e)^q}-\frac{1}{(\overline z+\e)^q}  \right)h^+~ \mathrm{d}x \leq 0.\]
Hence it must be that $z_\e \leq \overline z$ in $\Om$ for each $\e>0$. Now we use this in \eqref{wd2} and H\"{o}lder inequality to get
\[\|z_\e\|^2 \leq \la f(0)\int_{\Om}\overline z^{1-q}~\mathrm{d}x + \|\tilde f(u)\|_{L^2(\Om)}\|\overline z\|_{L^2(\Om)} := m_0 <+\infty\]
which implies $\limsup\limits_{\e>0}\|z_\e\| < +\infty$. Thus $\{z_\e\}_{\e>0}$ is a bounded sequence in $\tilde H^s(\Om)$ and so there must exist a $z \in \tilde H^s(\Om)$ such that, up to a subsequence, $z_\e \rightharpoonup z$ weakly in $\tilde H^s(\Om)$ as $\e\to 0$. We already know that $z_\e \to z$ pointwise a.e. in $\Om$. Moreover by Hardy's inequality, for any $\varphi \in \tilde H^s(\Om)$ we get
\[0 < \left| \frac{\varphi}{(z_\e+\e)^q} \right| \leq \left| \frac{\varphi}{(m\phi_{1,s})^q} \right| \in L^1(\Om). \]
Therefore we can use Lebesgue Dominated convergence theorem to pass through the limit as $\e \to 0^+$ in $(S_\la^\e)$ to obtain
 \begin{equation*}\label{weaksol2}
 C^n_s\int_Q\frac{(z(x)-z(y))(\varphi(x)-\varphi(y))}{|x-y|^{n+2s}}\mathrm{d}x\mathrm{d}y-\la f(0)\int_{\Om}\frac{\varphi}{z^q}\mathrm{d}x = \int_{\Om} \tilde f(u)\varphi\mathrm{d}x.
\end{equation*}
that is $z$ is a weak solution of $(S_\la)$. Finally it remains to show that $z \in C_{\phi_{1,s}}^+(\Om)$. But it easily following from $\bar{z}\geq z\geq z_\epsilon \geq m\phi_{1,s}$ in $\Om$. Thus, $T$ is well defined and this completes the proof. \hfill{\QED}
 \end{proof}

Before proving our next result, we recall Theorem $1.2$ from \cite{del-pezzo} as follows.
\begin{Theorem}\label{smp}
Let $c \in L^1_{loc}(\Om)$ be a non positive function and $ u \in H^s(\Om)$  be a weak supersolution of
\[(-\De)^su = c(x)u \; \text{in}\; \Om\]
then
\begin{enumerate}
\item If $\Om$ is bounded and $u\geq 0$ a.e. in $\mc C \Om$ then either $u>0$ a.e. in $\Om$ or $u =0$ a.e. in $\mb R^n$.
\item If $u \geq 0$ a.e. in $\mb R^n$ then either $ u>0$ a.e. in $\Om$ or $u=0$ a.e. in $\mb R^n$.
\end{enumerate}
\end{Theorem}
\begin{Lemma}\label{mi}
The map $T$ is strictly monotone increasing from $C_0(\overline \Om)$ to $C_{\phi_{1,s}}^+(\Om)$.
\end{Lemma}
\begin{proof}
First we show that $T$ is monotone increasing. For this, we let $u_1, u_2 \in C_0(\overline \Om)$ be such that $u_1\leq u_2$. Then $\tilde f(u_1)\leq \tilde f(u_2)$, since $\tilde f$ is increasing. Now let $z_i= T(u_i)$ for $i=1,2$, then each $z_i$ satisfies
\begin{equation*}
\begin{split}
(-\De)^sz_i -\la \frac{f(0)}{z_i^q} = \tilde f(u_i),\; z_i>0\; \text{in}\;\Om,\;\; z_i = 0 \; \text{in}\; \mb R^n \setminus \Om
\end{split}
\end{equation*}
and $\hat z:=(z_2-z_1)\in \tilde H^s(\Om)$ satisfies
\begin{equation}\label{mi1}
(-\De)^s (z_2-z_1)- \la f(0)\left( \frac{1}{z_2^q}- \frac{1}{z_1^q}\right)= \tilde f(u_2)-\tilde f(u_1)\geq 0\; \text{in}\; \Om.
\end{equation}
Then testing \eqref{mi1} with $\hat z^-$ gives
\begin{equation}\label{mi2}
\int_Q \frac{(\hat z(x)-\hat z(y))(\hat z^-(x)-\hat z^-(y))}{|x-y|^{n+2s}}~\mathrm{d}x\mathrm{d}y \geq \la f(0)\int_{\{z_2<z_1\}}\left( \frac{1}{z_2^q}- \frac{1}{z_1^q}\right)\hat z^-~\mathrm{d}x.
\end{equation}
It is easy to see that the right hand side of \eqref{mi2} is non positive and the left hand side can be estimated as
\[\int_Q \frac{(\hat z(x)-\hat z(y))(\hat z^-(x)-\hat z^-(y))}{|x-y|^{n+2s}}~\mathrm{d}x\mathrm{d}y \leq -\int_{\{z_2<z_1\}}\int_{\{z_2<z_1\}}\frac{|\hat z^-(x)-\hat z^-(y)|^2}{|x-y|^{n+2s}}~\mathrm{d}x\mathrm{d}y\leq 0.\]
Therefore it must be that $z^- =0$ in $\Om$ that is $z_2 \geq z_1$ in $\Om$. Now we assume that $ u_2 \geq u_1$ and $u_1 \not\equiv u_2$ then we show that $z_2>z_1$  in $\Om$. We already know that $z_2 \geq z_1$ and by Mean value theorem we get that there exists a $\xi \in (z_1,z_2)$ such that \eqref{mi1} can be written as
\begin{equation}\label{mi3}
(-\De)^s (z_2-z_1)+ \la f(0)\left( \frac{q}{\xi^{q+1}}\right)(z_2-z_1)= \tilde f(u_2)-\tilde f(u_1)\geq 0\; \text{in}\; \Om.
\end{equation}
Let $c(x)= \displaystyle \frac{1}{\xi^{q+1}(x)}$ then since $\xi \in (z_1,z_2)$ and $ z_i \in C_{\phi_{1,s}}^+(\Om)$, for $i=1,2$ we easily get that $c \in L^1_{loc}(\Om)$. Therefore from Theorem \ref{smp}, we obtain that $z_2-z_1>0$ in $\Om$. That is $T$ is a strictly monotone increasing map. \hfill{\QED}
\end{proof}

 The proof of our next result is motivated by the proof of Lemma $4.3$ in \cite{AJS}.

 \begin{Proposition}\label{comp}
 The map $T : C_{\phi_{1,s}}(\Om) \to C_{\phi_{1,s}}(\Om)$ is compact.
 \end{Proposition}
\begin{proof}
Let $ u \in C_{\phi_{1,s}}(\Om)$ and $T(u)=z \in C_{\phi_{1,s}}(\Om)$ then $z$ solves $(S_\la)$. We can write $z$ as
\[z = (\De)^{-s}\left( \la\frac{f(0)}{z^q}\right)+ (-\De)^{-s}(\tilde f(u))\; \text{in}\; \Om.\]
Let $\{u_k\}_{k \in \mb N} \subset C_{\phi_{1,s}}(\Om)$ be a bounded sequence that is $\sup\limits_{k \in \mb N}\left\|\displaystyle\frac{u_k}{\delta^s} \right\|_{L^\infty(\Om)}< +\infty$ and $T(u_k)= z_k \in C_{\phi_{1,s}}(\Om)$ for each $k$. Then we have
\[z_k = (\De)^{-s}\left( \la\frac{f(0)}{z_k^q}\right)+ (-\De)^{-s}(\tilde f(u_k))\; \text{in}\; \Om.\]
From the proof of Proposition \ref{well-defined} we infer that $m\phi_{1,s}$ and $Mw$ forms sub and supersolution of $(S_\la)$ respectively for appropriate choice of positive constants $m$ and $M$ (independent of $k$). Then by weak comparison principle we get that
\begin{equation}\label{comp1}
 m\phi_{1,s}\leq z_k \leq Mw \; \text{that is}\; k_1 \delta^s(x) \leq z_k(x) \leq k_2 \delta^s(x) \; \text{in}\; \Om
 \end{equation}
for some constants $k_1,k_2>0$.
In order to prove compactness of the map $T$, we need to show that the sequence $\{z_k\}$ is relatively compact in $C_{\phi_{1,s}}^+(\Om)$. Since $\{u_k\}$ is bounded in $C_{\phi_{1,s}}(\Om)$ we get $\tilde f(u_k) \in L^\infty(\Om)$ and $\sup\limits_{k \in \mb N}\|\tilde f(u_k)\|_{L^\infty(\Om)} \leq C_1$ for some constant $C_1>0$. Therefore from Theorem $1.2$ of \cite{RoSe} we obtain
\begin{equation}\label{comp2}
\left\| \frac{(-\De)^{-s} \tilde f(u_k)}{\delta^s}\right\|_{C^{0,\alpha}(\Om)} \leq C \|\tilde f(u_k)\|_{L^\infty(\Om)} \leq C_2
\end{equation}
for some constant $C, C_2>0$ (independent of $k$) and $0<\alpha< \min\{s,1-s\} $. Now for fix $\e>0$ we define the set
\[D_\e := \{x \in \Om: \; \delta(x)\geq \e\}\]
and let $\chi_{D_\e}$ denote the corresponding characteristic function on $D_\e$. We also define the following functions
\begin{equation*}
\begin{split}
z_k^{1,\e} &:= (\De)^{-s}\left( \la\frac{f(0)\chi_{D_\e}}{z_k^q}\right)+ (-\De)^{-s}(\tilde f(u_k))\\
z_k^{2,\e} &:= \left((\De)^{-s}\left( \la\frac{f(0)(1-\chi_{D_\e})}{z_k^q}\right) \right)\chi_{D_{3\e}}\\
z_k^{3,\e} &:= \left((\De)^{-s}\left( \la\frac{f(0)(1-\chi_{D_\e})}{z_k^q}\right)\right)(1-\chi_{D_{3\e}})
\end{split}
\end{equation*}
then clearly, $z_k = z_k^{1,\e}+ z_k^{2,\e}+ z_k^{3,\e}$. Therefore it is enough to prove that each $\{z_k^{i,\e}\}$ for $i=1,2,3$ is relatively compact in $C_{\phi_{1,s}}(\Om)$. Because of \eqref{comp1} we have
\[\la\frac{f(0)\chi_{D_\e}}{z_k^q} \leq \la\frac{f(0)\chi_{D_\e}}{k_1 \delta^{sq}(x)} \leq \la\frac{f(0)}{k_1 \e^{sq}} \]
which implies
\[\sup_{k \in \mb N}\left\| \la\frac{f(0)\chi_{D_\e}}{z_k^q} \right\|_{L^\infty(\Om)} \leq C_3= C_3(\e)\]
for some constant $C_3>0$. So from \eqref{comp2} and Theorem $1.2$ of \cite{RoSe} we infer that
\begin{equation}\label{comp3}
\left\|\frac{z_k^{1,\e}}{\delta^s}\right\|_{C^{0,\alpha}(\Om)} \leq C \left\| \la\frac{f(0)\chi_{D_\e}}{z_k^q} \right\|_{L^\infty(\Om)}+ \left\| \frac{(-\De)^{-s} \tilde f(u_k)}{\delta^s}\right\|_{C^{0,\alpha}(\Om)} \leq C_4 = C_4(\e)
\end{equation}
for some constant $C_4>0$. Thus for each fixed $\e>0$, $\{z_k^{1,\e}\}$ is relatively compact in $C_{\phi_{1,s}}(\Om)$. Considering the sequence $\{z_k^{2,\e}\}$, for any $x, x^\prime \in D_{3\e}$ we get
\begin{equation*}
\begin{split}
\left| \frac{z_k^{2,\e}(x)}{\delta^s(x)}- \frac{z_k^{2,\e}(x^\prime)}{\delta^s(x^\prime)} \right| &= \la \left|\int_\Om \left( \frac{G_s(x,y)}{\delta^s(x)}- \frac{G_s(x^\prime,y)}{\delta^s(x^\prime)} \right)\frac{f(0)(1-\chi_{D_\e})(y)}{z_k^q(y)} ~\mathrm{d}y \right|\\
& \leq C^\prime  \left|\int_\Om \left( \frac{G_s(x,y)}{\delta^s(x)}- \frac{G_s(x^\prime,y)}{\delta^s(x^\prime)} \right)\frac{(1-\chi_{D_\e})(y)}{k_1 \delta^{sq}(y)} ~\mathrm{d}y \right|
\end{split}
\end{equation*}
for some constant $C^\prime >0$. It has been proved in Lemma $4.3$ of \cite{AJS} that the map $x \mapsto \displaystyle \frac{G_s(x,y)}{\delta^s(x)}$ is H\"{o}lder continuous in $D_{3\e}$ uniformly with respect to $y \in \Om \setminus D_\e$ (but still depending on $\e$).  This implies that there exists $C_\e>0$ constant such that $\|G_s(x,y)\|_{C^s(D_{3\e})} \leq C_\e$ uniformly with respect to $y \in \Om\setminus D_\e$. Therefore we finally get that
\begin{align*}
\left| \frac{z_k^{2,\e}(x)}{\delta^s(x)}- \frac{z_k^{2,\e}(x^\prime)}{\delta^s(x^\prime)} \right| \leq \tilde C_\e|x-x^\prime|^s\int_{\Om \setminus D_\e} \frac{1}{\delta^{sq}(y)}~\mathrm{d}y \leq \hat C_\e |x-x^\prime|^s
\end{align*}
for some constant $\tilde C_\e, \; \hat C_\e >0$. This clearly gives that $\{z_k^{2,\e}\}$ is relatively compact in $C_{\phi_{1,s}}(\Om)$. Lastly we consider the sequence $\{z_k^{3,\e}\}$ and fix $\beta \in (sq,s)$. Recalling the estimate \eqref{esti-on-G} for $G_s(x,y)$, for $x \in \Om \setminus D_{3\e}$ we get
\begin{equation}\label{comp5}
\begin{split}
\left| \frac{z_k^{3,\e}(x)}{\delta^s(x)}\right| &= \left| \la \frac{f(0)}{\delta^s(x)}\int_{\mb R^n}\frac{G_s(x,y)(1-\chi_{D_\e})(y)}{\delta^s(x)z_k^q(y)}~\mathrm{d}y \right|\\
& \leq \left| \frac{\la f(0)}{\delta^s(x)} \int_{\mb R^n \setminus D_{\e}} \min\left( \frac{\delta^s(x)\delta^s(y)}{|x-y|^n}, \frac{\delta^s(x)}{|x-y|^{n-s}}\right) \frac{1}{k_1 \delta^{sq}(y)}~\mathrm{d}y \right|\\
& \leq \la f(0) \e^{\beta -sq} \int_{\mb R^n \setminus D_{\e}} \min\left( \frac{\delta^s(y)}{|x-y|^n}, \frac{1}{|x-y|^{n-s}}\right) \frac{1}{k_1 \delta^{\beta}(y)}~\mathrm{d}y \leq O(\e^{\beta-sq}).
\end{split}
\end{equation}
Now we show that $\left\{\displaystyle \frac{z_k}{\delta^s}\right\}$ is relatively compact in $L^\infty(\Om)$. Let $\tau >0$ be small enough. Then because of \eqref{comp5} we can always choose $\e$ small enough such that $\left\|\displaystyle\frac{z_k^{3,\e}}{\delta^s}\right\|_{L^\infty(\Om)} \leq \tau$. For each such $\e>0$ we can get a convergent subsequences $\{z_{k_m}^{1,\e}\}$ and $\{z_{k_m}^{2,\e}\}$ of $\{z_k^{1,\e}\}$ and $\{z_k^{2,\e}\}$ respectively in $L^\infty(\Om)$, since they are relatively compact in $C_{\phi_{1,s}}(\Om)$. Hence we have
\begin{align*}
\left\|\frac{z_{k_m}}{\delta^s} - \frac{z_{k_{m^\prime}}}{\delta^s} \right\|_{L^\infty(\Om)} \leq \left\|\frac{z^{1,\e}_{k_m}}{\delta^s} - \frac{z^{1,\e}_{k_{m^\prime}}}{\delta^s} \right\|_{L^\infty(\Om)}+ \left\|\frac{z^{2,\e}_{k_m}}{\delta^s} - \frac{z^{2,\e}_{k_{m^\prime}}}{\delta^s} \right\|_{L^\infty(\Om)} + 2\tau \leq 4\tau
\end{align*}
when $m,m^\prime \geq K$ for some $ K\in \mb N$. This implies that $\{z_{k_m}\}$ is a Cauchy sequence in $C_{\phi_{1,s}}(\Om)$ and hence convergent too. This proves that the sequence $\{z_k\}$ is relatively compact in $C_{\phi_{1,s}}(\Om)$. \hfill{\QED}
\end{proof}

 We seek help of solution to a nonlocal infinite semipositone problem (discussed in later section) for proving our next result that is the map $T$ is strongly increasing. By strongly increasing, we mean that if $u_1\leq u_2$ and $u_1 \not\equiv u_2$ then $T(u_2)-T(u_1)\in C_{\phi_{1,s}}^+(\Om)$.

\begin{Theorem}\label{si}
The map $T: C_{\phi_{1,s}}(\Om) \to C_{\phi_{1,s}}(\Om)$ is strongly increasing.
\end{Theorem}
\begin{proof}
Let $u_1\leq u_2$ such that $u_1 \not\equiv u_2$ and $T(u_i)= z_i$ for $i=1,2$. Then from Lemma \ref{mi} we already know that $z_1 > z_2$ in $\Om$ and $z_2 -z_1 \in C_{\phi_{1,s}}(\Om)$. So it remains to prove that there exist a $k_1>0$ such that $k_1\delta^s(x)\leq (z_2-z_1)(x)$ in $\Om$. We know that $(z_2-z_1)$ satisfies \eqref{mi3} and since $z_1(x)\leq \xi(x)\leq z_2(x) $ in $\Om$ and each $z_i \in C_{\phi_{1,s}}(\Om)$, we can get a constant $k>0$ such that $\displaystyle \la f(0)\frac{q}{\xi^{q+1}} \leq \frac{k}{\delta^{s(q+1)}(x)}$ in $\Om$. Therefore if we set $\tilde z= (z_2-z_1)$ then we obtain
\begin{equation}\label{si1}
(-\De)^s\tilde z+\frac{k}{\delta^{s(q+1)}(x)}\tilde z \geq 0, \; \tilde z>0\; \text{in}\; \Om,\;\; \tilde z=0 \;\text{in}\; \mb R^n \setminus \Om.
\end{equation}
From Theorem \ref{isp}, we know that for sufficiently small $\theta>0$, there exists a $v \in C_{\phi_{1,s}}^+(\Om)$ which satisfies weakly
\[(-\De)^s v = v^{p}-\frac{\theta }{v^{\gamma}}, \; v>0 \; \text{in}\; \Om,\;\; v=0\; \text{in}\; \mb R^n \setminus \Om\]
where $\gamma \in (q,1)$ and $p \in (0,1)$. So there exist constants $m_1,m_2>0$ such that $m_1\delta^s(x) \leq v(x)\leq m_2\delta^s(x)$ in $\Om$. From this we get
\begin{equation}\label{si2}
(-\De)^sv +\frac{kv}{\delta^{s(q+1)}(x)} = v^p - \frac{\theta}{v^{\gamma}(x)} + \frac{kv}{\delta^{s(q+1)}(x)} \leq v^p -\frac{m_2^{-\gamma}\theta}{\delta^{s\gamma}(x)}+ \frac{m_2k}{\delta^{sq}(x)}\; \text{in}\; \Om.
\end{equation}
Since $\gamma \in (q,1)$, the term $\displaystyle \frac{m_2\theta}{\delta^{s\gamma}(x)}$ dominates near the boundary of $\Om$. We define $\Om_\eta = \{x \in \Om:\;  \delta(x)< \eta\}$ and choose $\eta>0$ small enough so that \eqref{si2} gives
\begin{equation}\label{si3}
(-\De)^sv +\frac{kv}{\delta^{s(q+1)}(x)} \leq 0 ,\; v>0 \; \text{in}\; \Om_\eta,\;\; v=0\; \text{in}\; \mb R^n \setminus \Om
\end{equation}
From \eqref{si1} we have
\begin{equation}\label{si4}
(-\De)^s\tilde z +\frac{k\tilde z}{\delta^{s(q+1)}(x)} \geq 0 ,\; \tilde z>0 \; \text{in}\; \Om_\eta,\;\; \tilde z=0\; \text{in}\; \mb R^n \setminus \Om
\end{equation}
We choose $m_3>0$ small enough such that $m_3 v \leq \tilde z$ in $\Om \setminus \Om_\eta$. Thus from \eqref{si3} and \eqref{si4} gives
\begin{equation*}\label{si5}
(-\De)^s(m_3v-\tilde z) +\frac{k(m_3v-\tilde z)}{\delta^{s(q+1)}(x)} \leq 0\; \text{in} \; \Om_\eta,\; (m_3v-\tilde z)\leq 0 \; \text{in}\; \mb R^n \setminus \Om_\eta.
\end{equation*}
By comparison principle we get $m_1m_3 \delta^s(x)\leq m_3 v \leq \tilde z$ in $\Om_\eta$. Since $0 < \tilde z\in C_{\phi_{1,s}}(\Om)$, we get $\inf\limits_{x \in\Omega\backslash\Om_\eta} \tilde z>0$. Hence there must exist a constant $k_1>0$ such that $k_1\phi_{1,s}(x) \leq \tilde z $ in $\Om$. This proves that $(z_2-z_1)\in C_{\phi_{1,s}}^+(\Om)$ and the map $T$ is strongly increasing on $C_{\phi_{1,s}}(\Om)$.\hfill{\QED}
\end{proof}

We recall a fixed point theorem by Amann \cite{amann} which will help us to get the desired result.

\begin{Theorem}\label{fix-pt-thrm}
Let $X$ be a retract of some Banach space and $f : X\to X$ be a compact map. Suppose that $X_1$ and $X_2$ are disjoint subsets of $X$ and let $U_k$, $k=1,2$ be open subsets of $X$ such that $U_k\subset X_k$, $k=1,2$. Moreover, suppose that $f(X_k)\subset X_k$ and that $f$ has no fixed points on $X_k \setminus U_k$, $k=1,2$. Then $f$ has at least three distinct fixed points $x_1,x_2,x_3$ with $x_k\in X_k$, $k=1,2$ and $x \in X\setminus (X_1\cup X_2)$.
\end{Theorem}

We also recall Corollary $6.2$ of \cite{amann}.

\begin{Lemma}\label{cor6.2-amann}
Let $X$ be an ordered Banach space and $[y_1,y_2]$ be an ordered interval in $X$. Let $f:[y_1,y_2]\to X$ is an increasing compact map such that $f(y_1)\geq y_1$ and $f(y_2)\leq y_2$. Then $f$ has a minimal fixed point $\underline x$ and a maximal fixed point $\overline x$.
\end{Lemma}

Now the proof of the main result goes as follows.\\

\noi \textbf{Proof of Theorem \ref{3sol}:} To obtain solutions of $(P_\la)$ or  equivalently $(\tilde P_\la)$, it is enough to find fixed points of the map $T$, thanks to Proposition \ref{fixed-pt}. We define the sets $X= [\zeta_1,\vartheta_1]$, $X_1=[\zeta_1,\vartheta_2]$ and $X_2= [\zeta_2, \vartheta_1]$. Since $X$ and $X_i's$ for each $i=1,2$ are non empty closed and convex subsets of $C_{\phi_{1,s}}(\Om)$, they form retracts of $C_{\phi_{1,s}}(\Om)$. By construction (done in section 2), we know that $X_1\cap X_2 = \emptyset$ in $X$. Since $\zeta_1$ and $\vartheta_1$ are ordered sub and supersolutions of $(P_\la)$ respectively  and $T$ is strictly increasing (Lemma \ref{mi}), by comparison principle we obtain
\[\zeta_1 \leq T(\zeta_1)\leq T(\vartheta_1)\leq \vartheta_1.\]
This implies that $T(X) \subset X$ and similarly it also holds that $T(X_k)\subset X_k$ for $k=1,2$.  Because of Proposition \ref{comp} and Theorem \ref{si}, we get that $T: X\to X$ is compact and a strongly increasing map. It has been proved that $\vartheta_2$ is a strict supersolution of $(P_\la) $ and $T(\vartheta_2)\leq \vartheta_2$. So using Theorem \ref{smp} we infer that $T(\vartheta_2)<\vartheta_2$, $T(\vartheta_1)<\vartheta_1$, $T(\zeta_1)>\zeta_1$ and $T(\zeta_2)<\zeta_2$. Therefore Lemma \ref{cor6.2-amann} implies that $T$ has a maximal fixed point $u_1\in X_1$ such that $u_1 \in (\zeta_1,\vartheta_2)$ and a minimal fixed point $u_2\in X_2$ such that $u_2\in (\zeta_2, \vartheta_1)$. Now repeating the arguments in Theorem \ref{si}, we can prove that there exist constants $a_1,a_2>0$ such that
\[u_1\geq a_1\phi_{1,s}+\zeta_1,\;\vartheta_2 - u_1 \geq a_1 \phi_{1,s}, \;u_2+a_2\phi_{1,s} \leq \vartheta_1\; \text{and}\; u_2 -\zeta_2 \geq a_2 \phi_{1,s}\; \text{in}\; \Om.\]
We define the open ball $B$ in $X$ as
\[B:= X\cap \left\{\varphi \in C_{\phi_{1,s}}(\Om):\; \left\|\frac{\varphi}{\phi_{1,s}}\right\|_{L^\infty(\Om)}< a \right\} \;\text{with}\; a=\min(a_1, a_2).\]
Then for each $i=1,2$, $u_i+B\subset X_i$ and thus $X_i$'s have non empty interior. We construct open balls around each fixed point of $T$ in $X_i$ for each $i=1,2$ and take $U_i$ as the largest open set in $X_i$ containing all these open balls and such that $X_i \setminus U_i$ contains no fixed point of $T$. Now applying Theorem \ref{fix-pt-thrm}, we get the existence of third fixed point $u_3$ of $T$ lying in $X\setminus (X_1\cap X_2)$. This completes the proof.\hfill{\QED}

\section{Uniqueness for large $\la$ }
In this section, we prove that $(P_\la)$ admits a unique solution when $\la$ is sufficiently large.
We assume only that $f \in C^1([0,\infty))$ satisfies (f1)-(f3) and\\
(f5) there exists a $\alpha>0$ such that $\displaystyle \frac{f(u)}{u^q}$ is decreasing for $u >\alpha$.\\

\begin{Theorem}\label{exist-sol}
The problem $(P_\la)$ admits a solution for each $\la >0$.
\end{Theorem}
\begin{proof}
We recall the first pair of sub-supersolution $(\zeta_1,\vartheta_1)$ of $(P_\la)$ constructed in section 2. Without loss of generality, we can assume that $ f$ is non decreasing in $[\zeta_1, v_1]$. As in the proof of Theorem \ref{3sol}, we can show that if $X=[\zeta_1,\vartheta_1]$ then $T:X\to X$ is strictly increasing, compact map and $T(X)\subset X$. Therefore we can apply Lemma  \ref{cor6.2-amann} to conclude that $T$ has a fixed point $u_\la$ in $X$. Then Proposition \ref{fixed-pt} gives us that $u_\la\in C_{\phi_{1,s}}^+(\Om)$ is a solution of $(P_\la)$ which completes the proof.\hfill{\QED}
\end{proof}

\begin{Lemma}\label{lwr-bd}
Any positive solution $u_\la $ in $C_{\phi_{1,s}}^+(\Om)$ of $(P_\la)$ satisfies $u_\la \geq \Theta_\la w$ in $\Om$ where $\Theta_\la = (\la f(0))^{\frac{1}{1+q}}$ and $w$ is the solution to \eqref{psp}.
\end{Lemma}
\begin{proof}
Let $u_\la$ solves $(P_\la)$. Since $f$ is nondecreasing, we get
\begin{equation*}
(-\De)^su_\la(x) = \frac{\la f(u_\la(x))}{u_\la^q(x)} \geq \frac{\la f(0)}{u_\la^q(x)}\quad \mbox{in }\Omega
\end{equation*}
and
\begin{equation*}
(-\De)^s (\Theta_\lambda w)(x) =  \frac{\la f(0)}{(\Theta_\lambda w)^q(x)}\quad \mbox{in }\Omega.
\end{equation*}
 Therefore by weak comparison principle (see Lemma~\ref{mi}), we conclude that $u_\la- \Theta_\lambda w\geq 0$ in $\mb R^n$ .\hfill{\QED}
\end{proof}

\begin{Corollary}\label{min-sol}
There exist a minimal solution of $(P_\la)$ in $C_{\phi_{1,s}}^+(\Om)$, for each $\la>0$.
\end{Corollary}
\begin{proof}
From Theorem \ref{exist-sol} we know that $(P_\la)$ has a solution $u_\la \in C_{\phi_{1,s} }^+(\Om)$ such that
\[\zeta_1 \leq u_\la \leq \vartheta_1\;\text{in}\; \Om\]
where both $\zeta_1,\vartheta_1 \in C_{\phi_{1,s}}^+(\Om)$. Now the result follows from Lemma~\ref{lwr-bd} and Lemma~\ref{cor6.2-amann}.\hfill{\QED}
\end{proof}

\begin{Theorem}\label{uniq}
There exist a $\lambda^*>0$ such that $(P_\la)$ has a unique solution when $\la > \la^*$.
\end{Theorem}
\begin{proof}
Let $u_\la$ and $\bar u_\la$ be two distinct positive solutions of $(P_\la)$ in $C_{\phi_{1,s}}^+(\Om)$ such that $u_\la$ is the minimal solution as obtained from Corollary~\ref{min-sol}. So it holds that $u_\la \leq \bar u_\la $ in $\Om$.
We have
\begin{equation}\label{uniq1}
\int_Q(-\De)^s (u_\la -\bar u_\la)(u_\la -\bar u_\la)~\mathrm{d}x = \la \int_\Om\left( \frac{f(u_\la)}{u_\la^q} - \frac{f(\bar u_\la)}{\bar u_\la^q} \right) (u_\la -\bar u_\la)~\mathrm{d}x
\end{equation}
which gives
\begin{equation}\label{uniq2}
C^n_s\|(u_\la -\bar u_\la)\|^2= \la \int_\Om\left( \frac{f(u_\la)}{u_\la^q} - \frac{f(\bar u_\la)}{\bar u_\la^q} \right) (u_\la -\bar u_\la)~\mathrm{d}x= \la \int_\Om \left( \int_0^1 f_0^\prime(u_\la + t(\bar u_\la-u_\la))~\mathrm{d}t\right)(u_\la -\bar u_\la)^2~\mathrm{d}x
\end{equation}
where $f_0(u)= \displaystyle \frac{f(u)}{u^q} $. From Lemma \ref{lwr-bd} we know that $u_\la \geq \Theta_\la w \geq \Theta_\la k_1\delta^s(x)$ in $\Om$ for some $k_1>0$. So if we define
\[\Om_0 :=\left\{x \in \Om:\; f_1(x)\geq 0\right\}\subset \left\{x \in \Om:\; \delta^s(x)\leq \frac{\alpha}{\Theta_\lambda k_1}\right\}\]
where $f_1(x)=\displaystyle\int_0^1 f_0^\prime(u_\la + t(\bar u_\la-u_\la))~\mathrm{d}t$, then $u_\lambda(x)\geq \alpha$ in $\Omega_0$.
From \eqref{uniq2} we obtain
\[C^n_s\|(u_\la -\bar u_\la)\|^2= \la \left( \int_{\Om_0}f_1(x)(u_\la -\bar u_\la)^2~\mathrm{d}x+  \int_{\Omega\backslash\Om_0}f_1(x)(u_\la -\bar u_\la)^2~\mathrm{d}x\right).\]
Since $f_1 \leq 0$ in $\Omega\backslash\Om_0$,
\begin{equation}\label{uniq2bis}
C^n_s\|(u_\la -\bar u_\la)\|^2 \leq \la  \int_{\Om_0}f_1(x)(u_\la -\bar u_\la)^2~\mathrm{d}x.
\end{equation}
We also notice that $\lim\limits_{\la\to +\infty}\delta(x)=0$ for $x \in \Om_0$ and since by (f5) $u_\lambda+t(\bar u_\lambda-u_\lambda)(x)\leq \alpha$ for $x\in\Om_0$, there exists a $M_2>0$ such that $|f^\prime((u_\la + t(\bar u_\la-u_\la))(x))|\leq M_2$ for all $x \in \Om_0$. Therefore we also have the following estimate using Hardy's inequality
\begin{align*}
 \la\int_{\Om_0}f_1(x)(u_\la -\bar u_\la)^2(x)~\mathrm{d}x &\leq \la \int_{\Om_0}\left( \int_0^1 \frac{f^\prime((u_\la + t(\bar u_\la-u_\la))(x))}{(u_\la + t(\bar u_\la-u_\la))^q(x)}~\mathrm{d}t\right)(u_\la -\bar u_\la)^2(x)~\mathrm{d}x\\
 & \leq \la M_2 \int_{\Om_0} \frac{(u_\la -\bar u_\la)^2(x)}{u_\la^q(x)}~\mathrm{d}x \\
 & \leq \la M_2 (\Theta_\la k_1)^{-q}\int_{\Om_0}\frac{(u_\la -\bar u_\la)^2(x)\delta^{(2-q)s}(x)}{\delta^{2s}(x)}~\mathrm{d}x , \; \text{since}\; u_\la \geq \Theta_\la w\\
 &\leq \la M_2 (\Theta_\la k_1)^{-q} \left( \frac{\sigma}{\Theta_\la k_1}\right)^{2-q} \int_{\Om_0}\frac{(u_\la -\bar u_\la)^2(x)}{\delta^{2s}(x)}~\mathrm{d}x\\
 & \leq C(\lambda)\|(u_\la-\bar u_\la)\|^2
 \end{align*}
 where $C(\lambda)=O(\lambda^{-\frac{(1-q)}{1+q}})$. This gives a contradiction for $\la$ large enough since $q \in (0,1)$. Therefore we state that $u_\la \equiv \bar u_\la$ when $\la$ is sufficiently large and this completes the proof.\hfill{\QED}
\end{proof}

\section{A fractional and singular semipositone problem}
 We devote this section to prove the existence of weak solution for the following nonlocal infinite semipositone problem
 \begin{equation*}
 (I_\theta):\;\;\;\;(-\De)^s v=v^p-\frac{\theta}{v^{\gamma}},\;v >0,\; \text{in}\; \Om, \;\;v=0 \; \text{in}\; \mb R^n \setminus \Om
 \end{equation*}
where $p \in (0,1)$, $\theta$ is a positive parameter and $\gamma \in (q,1)$. Before this we consider the following problem
 \begin{equation*}
 (I_0):\;\;\;\; (-\De)^s v=v^p,\;v >0,\; \text{in}\; \Om, \;\;v=0 \; \text{in}\; \mb R^n \setminus \Om
 \end{equation*}
 for $p \in (0,1)$. The energy functional  $E_0 : \tilde H^s(\Om) \to \mb R$ associated to $(I_0)$ is given by
 \[E_0(v):= C^n_s\frac{\|v\|^2}{2} -\frac{1}{p+1}\int_{\Om}|v|^{p+1}~\mathrm{d}x \]
 for $v \in \tilde H^s(\Om)$. Then $E_0$ is weakly lower semicontinuous and coercive which implies that $E_0$ possesses a global minimizer say $v_0 \in \tilde H^s(\Om)$. Since $\inf\limits_{\tilde H^s(\Om)} E_0<0$ and $E_0(|v_0|)\leq E_0(v_0)$, we get $v_0 \not\equiv 0$ and we can assume that $v_0\geq 0$ in $\Om$. Now it is easy to see that $v_0$ solves the problem $(I_0)$ weakly. From Proposition $2.2$ of \cite{barrios}, we say that $v_0\in L^\infty(\Om)$ and then using Theorem $1.2$ of \cite{RoSe} we conclude that $v_0\in C^s(\mb R^n)\cap C_{\phi_{1,s}}(\Om)$. Now by strong maximum principle it follows that $v_0 >0$ in $\Om$. For $\eta >0$ small enough, $\eta\phi_{1,s}$ forms a subsolution of $(I_0)$ and then it is easy to show by weak comparison principle that $v_0\in C_{\phi_{1,s}}^+(\Om)$.
   The uniqueness of $v_0$ as a solution of $(I_0)$ follows by using the Picone identity (Lemma $6.2$ of \cite{Pi-id}) and following the arguments in Theorem $5.2$ in \cite{TJS-para}.\\
   For fix $\mu>0$, let us consider the solution operator $ G(\theta,v): \{|\theta|< \mu\}\times B_\e(v_0) \to C_{\phi_{1,s}}(\Om)$ defined as
   \[G(\theta,v):= v- (-\De)^{-s}\left(v^p-\frac{\theta}{v^\gamma} \right)\]
for $(\theta,v)\in \{|\theta|< \mu\}\times B_\e(v_0)$, where $B_\e(v_0)$ denotes the open ball in $C_{\phi_{1,s}}^+(\Om)$ with center $v_0$ and radius $\e>0$. We point out that for $\epsilon>0$ small enough, $B_\epsilon(v_0)\subset C_{\phi_{1,s}}^+(\Om)$. Furthermore, $G(\theta, u)=0$ if and only if $u$ solves $(I_\theta)$. Let $(\theta,v)\in \{|\theta|< \mu\}\times B_\e(v_0)$ then $(-\De)^{-s}v^p \in C_{\phi_{1,s}}(\Om)$, by Theorem $1.2$ of \cite{RoSe}. Since $v\in L^\infty(\Omega)$, we have that $(-\De)^{-s}v^p \in C_{\phi_{1,s}}(\Om)$. Moreover $(-\De)^{-s}\left(\displaystyle\frac{\theta}{v^{\gamma}}\right)\in C_{\phi_{1,s}}(\Om)$ follows from Theorem $1.2$ in \cite{AJS} and the fact that $v\in C_{\phi_{1,s}}^+(\Om)$. Thus the map $G(\cdot,\cdot)$ is well defined.

\begin{Lemma}\label{cont-singprob}
The map $G$ is continuously Fr\'{e}chet differentiable.
\end{Lemma}
\begin{proof}
We begin with showing that $G$ is continuous. Let $v,v_k \in B_\e(v_0)$, $|\theta|<\mu$ and $\tau\in \mb R^n$ be such that $(\|v_k-v\|_{C_{\phi_{1,s}}(\Om)}+|\tau| ) \to 0$ as $k \to \infty$, then
\begin{align*}
|G(\theta+\tau,v_k)-G(\theta,v)| &= \left|(v_k-v) - (-\De)^{-s}(v_k^q-v^q) + (-\De)^{-s}\left( \frac{\theta+\tau}{v_k^{\gamma}}- \frac{\theta}{v^{\gamma}}\right)\right|\\
& \leq \|(v_k-v)\|_{C_{\phi_{1,s}}(\Om)}\phi_{1,s}+ C_1\|(v_k^q-v^q)\|_{C_{\phi_{1,s}}(\Om)}\phi_{1,s}\\
& \quad \quad+ \theta \left|(-\De)^{-s}\left(\frac{\gamma(v_k-v)}{(v+\xi (v_k-v))^{\gamma+1}} + \frac{\tau}{v_k^\gamma} \right)\right|\\
& \leq \|(v_k-v)\|_{C_{\phi_{1,s}}(\Om)}\phi_{1,s}(x)+ C_1\|(v_k-v)\|_{C_{\phi_{1,s}}(\Om)}\phi_{1,s}\\
& \quad \quad+ C_3\theta \left|(-\De)^{-s}\left(\frac{\gamma\|(v_k-v)\|_{C_{\phi_{1,s}}(\Om)}}{\delta^{s\gamma}(x)} + \frac{\tau}{\delta^{s\gamma}(x)} \right)\right|
\end{align*}
for appropriate constants $C_i>0$, $i=1,2,3$. Now Proposition $1.2.9$ of \cite{abatangelo} gives that
\[|G(\theta+\tau,v_k)-G(\theta,v)| \leq O(\|v_k-v\|_{C_{\phi_{1,s}}(\Om)}+|\tau| )\phi_{1,s}\]
which implies that $G$ is continuous on $\{|\theta|< \mu\}\times B_\e(v_0)$. Following similar arguments we can show that
\[\lim_{t\to 0^+}\frac{G(\theta, v+t\phi)-G(\theta,v)}{t}= \phi- p(-\De)^{-s}(v^{p-1}\phi) - \theta (-\De)^{-s}(\gamma v^{-\gamma-1}\phi) \]
for $v, \phi \in B_\e(v_0)$. This implies that $G(\theta,\cdot)$ is Gat\'{e}aux differentiable and
\[D_vG(\theta,v)(\phi)= \phi- p(-\De)^{-s}(v^{p-1}\phi) - \theta (-\De)^{-s}(\gamma v^{-\gamma-1}\phi).
\]
Next, to prove that $G$ is Fr\'{e}chet differentiable we first consider
\begin{align*}
&|G(\theta,v+\phi)-G(\theta,v)-D_vG(\theta,v)(\phi)|\\
&= \left|-(-\De)^{-s}\left((v+\phi)^p-v^p-pv^{p-1}\phi \right) + \theta(-\De)^{-s}\left( \frac{1}{(v+\phi)^\gamma}-\frac{1}{v^\gamma}+ \frac{\gamma\phi}{v^{\gamma+1}} \right) \right|.
\end{align*}
Since $v \in C_{\phi_{1,s}}^+(\Om)$ and $\phi\in C_{\phi_{1,s}}(\Om)$, we get $\left|\frac{\phi}{v}\right|\leq K$ for some constant $K\geq 0$. This along with Taylor series expansion gives for some $\theta_0\in (0,1)$
\[\left|(v+\phi)^p-v^p-pv^{p-1}\phi= \frac{p(p-1)\phi^2}{2(v+\theta_0\phi)^{2-p}}\right| \leq C\|\phi\|_{C_{\phi_{1,s}}(\Om)}^2.\]
So applying Theorem $1.2$ of \cite{RoSe}, we get that
\[|(-\De)^{-s}\left((v+\phi)^p-v^p-pv^{p-1}\phi \right)| \leq O\left(\|\phi\|^2_{C_{\phi_{1,s}}(\Om)}\right).\]
Also similar idea gives, for some $\xi_1\in (0,1)$
\begin{align*}
\frac{1}{(v+\phi)^\gamma}-\frac{1}{v^\gamma}+ \frac{\gamma\phi}{v^{\gamma+1}} &= -\frac{\gamma \phi}{(v+\xi_1\phi)^{\gamma+1}}+ \frac{\gamma\phi}{v^{\gamma+1}}
 = \gamma \phi \left( \frac{(v+\xi_1\phi)^{\gamma+1} - v^{\gamma +1}}{(v(v+\xi_1\phi))^{\gamma+1}} \right)\\
 & \leq C_4 \frac{\|\phi\|_{C_{\phi_{1,s}}(\Om)}^2}{\delta^{s\gamma}}
\end{align*}
for appropriate constant $C_4>0$. So again using Proposition $1.2.9$ of \cite{abatangelo} we get
\[\left|(-\De)^{-s}\left( \frac{1}{(v+\phi)^\gamma}-\frac{1}{v^\gamma}+ \frac{\gamma\phi}{v^{\gamma+1}} \right) \right| \leq O\left(\|\phi\|_{C_{\phi_{1,s}}(\Om)}^2\right). \]
Therefore we get
\[\|G(\theta,v+\phi)-G(\theta,v)-D_vG(\theta,v)(\phi)\|_{C_{\phi_{1,s}}(\Om)} \to 0\; \text{as}\; \|\phi\|_{C_{\phi_{1,s}}(\Om)} \to 0. \]
Now we prove continuity of $D_vG(\theta,v)$. Consider $\{v_k\}_{k\in\mb N}\subset B_\epsilon(v_0)$ such that $\|v_k-v\|_{C_{\phi_{1,s}}(\Om)} \to 0$ as $k \to \infty$ then
\begin{align*}
\|D_vG(\theta,v_k)- D_vG(\theta,v)\| = \sup_{0\neq \phi \in C_{\phi_{1,s}}(\Om)}\frac{\|(D_vG(\theta,v_k)- D_vG(\theta,v))\phi\|_{C_{\phi_{1,s}}(\Om)}}{\|\phi\|_{C_{\phi_{1,s}}(\Om)}}.
\end{align*}
We have that
\begin{align*}
(D_vG(\theta,v_k)- D_vG(\theta,v))\phi = -p(-\De)^{-s}((v_k^{p-1}-v^{p-1})\phi)-\gamma \theta(-\De)^{-s}\left( \frac{\phi}{v_k^{\gamma+1}}- \frac{\phi}{v^{\gamma+1}}\right)
\end{align*}
But using $\phi \in C_{\phi_{1,s}}(\Om)$ and again using similar arguments as before we get
\[|(D_vG(\theta,v_k)- D_vG(\theta,v))\phi| \leq O\left(\|\phi\|_{C_{\phi_{1,s}}(\Om)}(\|v_k^{p-1}-v^{p-1}\|_{C_{\phi_{1,s}}(\Om)}+\|v_k-v\|_{C_{\phi_{1,s}}(\Om)})\right)   \]
which implies that $D_vG(\theta,v)$ is continuous. Similarly we can prove that $D_\theta(\theta,v)$ exists and is continuous where $D_\theta(\theta,v)= (-\De)^{-s}\left( \frac{1}{v^\gamma}\right)$. \hfill{\QED}
\end{proof}

The linearisation of the map $G$ with respect to the second variable at $(\theta, v)\in \{|\theta|<\mu\} \times B_\e(v_0)$ given by $\partial_vG(\theta,v): C_{\phi_{1,s}}^+(\Om) \to C_{\phi_{1,s}}(\Om)$ is defined as
\[ \partial_vG(\theta,v)h = h - (-\De)^{-s}\left( pv^{p-1}h+ \frac{\theta \gamma h}{u^{\gamma+1}} \right),\; \text{for}\; h \in C_{\phi_{1,s}}(\Om). \]
 Since $v_0$ solves $(I_0)$, clearly $G(0,v_0)=0$. Now we will show that the map $h \mapsto \partial_vG(0,v_0)h$ is invertible. To do this, we will be studying an eigenvalue problem. Let us define
 \[\Lambda := \inf_{0\not\equiv u \in \tilde H^s(\Om)} \left(\frac{\|u\|^2 - p\int_\Om v_0^{p-1}u^2 ~\mathrm{d}x}{\int_\Om u^2 ~\mathrm{d}x}\right) = \inf_S \left(\|u\|^2 - p\int_\Om v_0^{p-1}u^2 ~\mathrm{d}x\right) \]
where $S= \{u\in \tilde H^s(\Om), \int_\Om u^2 ~\mathrm{d}x =1\}$. Then by Hardy's inequality and $v_0\in C_{\phi_{1,s}}^+(\Om)$ it follows that
\[\int_{\Om}v_0^{p-1}u^2~\mathrm{d}x \leq k_2\int_\Om \frac{u^2}{\delta^{2s}(x)}\delta^{s(p+1)}~\mathrm{d}x < +\infty.\]
 So the functional
 \[I_{00}(u)= \|u\|^2 - p\int_\Om v_0^{p-1}u^2 ~\mathrm{d}x \]
 is well defined on $\tilde H^s(\Om)$. Following standard minimization  arguments and using compact embedding of $\tilde H^s(\Om)$ in $L^2(\Om)$, it is easy to show that $\inf I_{00}(S)= I_{00}(\psi)= \Lambda$ for some $\psi \in S$. Also since $I_{00}(|\psi|)\leq I_{00}(\psi)$, by minimality of $\psi$ we assert that without loss of generality we may assume that $\psi \geq 0$. Then $\psi$ satisfies
 \begin{equation}\label{ef1}
 (-\De)^s \psi = \Lambda \psi +p v_0^{p-1}\psi\; \text{in}\; \Om,\;\; \psi=0 \;\text{in}\; \mb R^n \setminus \Om
 \end{equation}
which implies that $\psi$ is an eigenfunction corresponding to the eigenvalue $\Lambda$ for the operator $(-\De)^s- pv_0^{p-1}$ with homogenous Dirichlet boundary condition in $\Om$. We obtain $\psi \in L^\infty(\Om)$ by following the arguments in Theorem $3.2$ (p. 379) of \cite{palatucci}. So local regularity results assert that $\psi \in C^s_{\text{loc}}(\Om)$. We claim that $\psi>0$ in $\Om$ because if is not true, then there exist  a $x_0\in K \Subset\Om$ such that $\psi(x_0)=0$. But this gives
\[0 > 2C^n_s\int_{\mb R^n} \frac{(\psi(x_0)-\psi(y))}{|x-y|^{n+2s}}~\mathrm{d}y = \Lambda \psi + pv_0^{p-1}(x_0)\psi(x_0) =0\;\text{in}\; K \]
which is  a contradiction. Therefore $\psi >0$ in $\Om$.\\
\textbf{Claim(1):} $\Lambda$ is a principal eigenvalue.\\
We have to show that any eigenfunction (say $\psi$) associated to it does not change sign. Assume by contradiction that $\psi^+\not\equiv 0$ and $\psi^-\not\equiv 0$. Then
\begin{equation*}
\begin{split}
&C^n_s \int_{Q} \frac{(\psi(x)-\psi(y))(\psi^+(x)-\psi^+(y))}{|x-y|^{n+2s}}~\mathrm{d}y\mathrm{d}x\\ &= C^n_s \left(\int_{Q} \frac{(\psi^+(x)-\psi^+(y))(\psi^+(x)-\psi^+(y))}{|x-y|^{n+2s}}~\mathrm{d}y\mathrm{d}x+ 2\int_Q \frac{\psi^-(x)\psi^+(y)}{|x-y|^{n+2s}}~\mathrm{d}y\mathrm{d}x \right)\\
& = \Lambda \int_\Om \psi\psi^+~\mathrm{d}x + p\int_\Om v_0^{p-1}\psi\psi^+\mathrm{d}x.
\end{split}
\end{equation*}
Now if $\psi^+, \psi^-\not\equiv 0$ then
\[\|\psi^+\|^2 - p\int_\Om v_0^{p-1}(\psi^+)^2\mathrm{d}x = \Lambda \int_\Om (\psi^+)^2~\mathrm{d}x -  2\int_Q \frac{\psi^-(x)\psi^+(y)}{|x-y|^{n+2s}}~\mathrm{d}y\mathrm{d}x  < \Lambda \int_\Om (\psi^+)^2~\mathrm{d}x \]
but this contradicts the definition of $\Lambda$. This proves the claim.\\
\textbf{Claim(2):} $\Lambda $ is the unique principal eigenvalue.\\
Suppose not, that is let $\Lambda_0$ is another principal eigenvalue of $(-\De)^s- pv_0^{p-1}$ and $\psi_0\in \tilde H^s(\Om)$ denotes the corresponding eigenfunction. Then $\psi_0>0$ in $\Om$ and satisfies
 \begin{equation}\label{ef2}
  (-\De)^s \psi_0 = \Lambda_0 \psi_0 +p v_0^{p-1}\psi_0\; \text{in}\; \Om,\;\; \psi_0=0 \;\text{in}\; \mb R^n \setminus \Om.
 \end{equation}
 Testing \eqref{ef1} with $\psi_0$ and \eqref{ef2} with $\psi$ gives
 \[\Lambda \int_\Om \psi\psi_0~\mathrm{d}x= \Lambda_0 \int_\Om \psi_0\psi~\mathrm{d}x.\]
 Since $\psi,\psi_0>0$ in $\Om$, we get $\Lambda=\Lambda_0$.\\
\textbf{Claim(3):} Any nonnegative eigenfunction $\psi \in C_{\phi_{1,s}}^+(\Om)$.\\
For sufficiently small $\eta>0$, using Theorem~\ref{smp} and Hopf Lemma, it is easy to show that $\psi \geq \eta\phi_{1,s}$. Assume fisrt that $\Lambda \geq 0$ then
\[\psi= \Lambda (-\De)^{-s}\psi + p(-\De)^{-s}\left( \frac{\psi}{v_0^{1-p}}\right)\; \text{in}\; \Om.\]
Since $\psi \in L^\infty(\Om)$ we get $(-\De)^{-s}\psi  \in C_{\phi_{1,s}}(\Om)$. Also since $v_0\in C_{\phi_{1,s}}^+(\Om)$ we get that $$\displaystyle \frac{\psi}{v_0^{1-p}} \sim   \frac{C}{\delta^{s(1-p)}(x)}\; \text{in}\; \Om$$ for some $C>0$. So $(-\De)^{-s}\left(\displaystyle \frac{\psi}{v_0^{1-p}}\right) \in C_{\phi_{1,s}}^+(\Om)$, follows from Proposition $1.2.9$ of \cite{abatangelo} together with $s(1-p) < s$. Therefore $\psi \in C_{\phi_{1,s}}^+(\Om)$ when $\Lambda\geq 0$. In the other case, if $\Lambda<0$ then $\psi$ satisfies
\[(-\De)^s \psi +(-\Lambda) \psi =  p\left( \frac{\psi}{v_0^{1-p}}\right)\sim \frac{C}{\delta^{s(1-p)}(x)}\geq 0\; \text{in}\; \Om.\]
So by using Theorem $1.5 (1)$ of \cite{del-pezzo} and Theorem $1.2$ of \cite{RoSe}, we infer that $\psi \in C_{\phi_{1,s}}^+(\Om)$. This proves the claim.\\
\textbf{Claim(4):} $\Lambda>0$.\\
First we show that $\Lambda$ is non zero. Suppose it is equal to zero then \eqref{ef1} reduces to
\[(-\De)^s \psi = pv_0^{p-1}\psi\; \text{in}\; \Om.\]
Using $v_0$ as a test function in the expression above gives
\[\int_{\mb R^n}(-\De)^s\psi v_0~\mathrm{d}x = p\int_\Om v_0^p\psi~\mathrm{d}x . \]
But we know that $v_0$ is a unique solution of $(I_0)$. Therefore we get
\[p\int_\Om v_0^p\psi~\mathrm{d}x = \int_\Om v_0^p\psi~\mathrm{d}x\]
which gives $p=1$ since $\psi,v_0>0$ in $\Om$. This gives a contradiction, thus $\Lambda \neq 0$. Now we assume by contradiction that $\Lambda <0$. For $\epsilon>0$ small enough, we consider the function $v_0 - \e\psi$. Then since $\Lambda <0$ and $p-1<0$, we get that in $\Om$
\[(-\De)^s(v_0 -\e\psi) = v_0^p -\e \Lambda \psi - \e pv_0^{p-1}\psi > v_0^p  - \e pv_0^{p-1}\psi \geq (v_0-\e\psi)^p. \]
This implies that $v_0 - \e\psi$ forms a strict supersolution of $(I_0)$. It is already known that $\eta\phi_{1,s}$ for sufficiently small choice of $\eta$ forms a subsolution of $(I_0)$. Therefore there must be a function $\hat v\in \tilde H^s(\Om) $ such that $\eta \phi_{1,s}\leq \hat v\leq (v_0-\e\psi)$ which is a solution of $(I_0)$. But this contradicts the uniqueness of $v_0$ due to $\e\psi>0$ in $\Om$. Hence $\Lambda>0 $.

\begin{Theorem}\label{isp}
For a small range of $\theta$, the problem $(I_\theta)$ admits a solution.
\end{Theorem}
\begin{proof}
We have already proved that $G$ is a continuously Fr\'{e}chet differentiable map. Now we consider the problem
\begin{equation}\label{isp1}
(-\De)^s u -pv_0^{p-1}u = \psi, \; u>0 \;\text{in}\; \Om,\;\; u=0\;\text{in}\; \mb R^n\setminus \Om.
\end{equation}
By defining the energy functional corresponding to it and minimization arguments, it is easy to show that the above problem has a solution $u\in \tilde H^s(\Om)$. Now suppose $u_1,u_2\in \tilde H^s(\Om)$ be two distinct solutions of \eqref{isp1} then
\[0 = \int_{\mb R^n}(-\De)^s(u_1-u_2)(u_1-u_2)~\mathrm{d}x- p\int_\Om v_0^{p-1}(u_1-u_2)^2~\mathrm{d}x \geq \Lambda \int_{\Om}(u_1-u_2)^2~\mathrm{d}x>0 \]
since $\Lambda>0$. This implies that the solution must be unique. Also using similar argument as in Claim(3) gives $u \in C_{\phi_{1,s}}(\Om)$. All this along with previous claims guarantees that the map $ \partial_vG(0,v_0): C_{\phi_{1,s}}(\Om) \to C_{\phi_{1,s}}(\Om) $ is invertible. Hence we apply implicit function theorem to get that there exist a subset $\{|\theta|< \mu^\prime\}\times B_{\e^\prime}(v_0) \subset \{|\theta|< \mu\}\times B_{\e}(v_0)$ for $\mu^\prime <\mu$ and $\e^\prime <\e$ and a $C^1$-map $h\,:\,\{|\theta|<\mu'\}\mapsto B_{\epsilon'}(v_0)$ such that $G(\theta, v)=0$ in $\{|\theta|< \mu^\prime\}\times B_{\e^\prime}(v_0)$ coincides with the graph of $h$. This completes the proof. \hfill{\QED}
\end{proof}

\medskip

\linespread{0.5}

\noindent {\bf Acknowledgements:} The authors were partially funded by IFCAM (Indo-French Centre for Applied Mathematics) UMI CNRS 3494 under the project "Singular phenomena in reaction diffusion equations and in conservation laws".


\begin{thebibliography}{21}
\linespread{0.1}

\bibitem{abatangelo} N. Abatangelo, {\it Large S-harmonic functions and boundary blow-up solutions for the fractional Laplacian}, Discrete Contin. Dyn. Syst., 35(12) (2015), 5555-5607.

\bibitem{AJS} {Adimurthi, J. Giacomoni and S. Santra, {\it Positive solutions to a fractional equation with singular nonlinearity}, submitted, preprint available at https://arxiv.org/pdf/1706.01965.pdf}

\bibitem{amann} H. Amann, {\it Fixed point equations and nonlinear eigenvalue problems in ordered Banach spaces}, SIAM Rev., 18 (1976), 620--709.

\bibitem{Pi-id} S. Amghibech, {\it On the discrete version of picone's identity}, Discrete Applied Mathematics, 156 (2008), 1--10.

\bibitem{da} D. Applebaum, {\it L$\acute{e}$vy Processes-From Probability to Finance and Quantum Groups}, Notices Amer. Math. Soc.,
51 (11) (2004), 1336-1347.

\bibitem{peral}  B. Barrios, I. De Bonis,  M. Medina and  I. Peral, {\it  Semilinear problems for the fractional laplacian with a singular nonlinearity}, Open Math., 13 (2015), 390-407.

\bibitem{barrios} B. Barrios, E. Colorado, R. Servadei and F. Soria, {\it A critical fractional equation with concave-convex power
nonlinearities}, Ann. I. H. Poincar\'{e}  A.N., 32 (2015), 875-900.

\bibitem{CS}  L. Caffarelli, L. Silvestre, {\it An extension problem related to the fractional Laplacian}, Comm.
    Partial Differential Equations, 32 (2007), 1245-1260.

\bibitem{CES} A. Castro, E. Ko and R. Shivaji, {\it A uniqueness result for a singular nonlinear
eigenvalue problem}, Proceedings of the Royal Society of Edinburgh, 143A (2013), 739-744.

\bibitem{Chen-song} Z-Q.Chen and R. Song, {\it Estimates on Green functions and Poissons kernels for symmetric stable processes},
Math. Ann., 312 (1998), 465--501.

\bibitem{crt} M. G. Crandall,  P. H. Rabinowitz and L. Tartar, {\it On a Dirichlet problem with a singular nonlinearity}, Communications in Partial Differential Equations, 2 (1977), 193-222.

\bibitem{del-pezzo} L. M.DelPezzo and A. Quaas, {\it A Hopf lemma and a strong minimum principle for the fractional p-Laplacian}, J. Differential Equations, 263 (2017), 765-778.

\bibitem{Dhanya-JMAA} R. Dhanya, Eunkyung Ko and  R. Shivaji, {\it A three solution theorem for singular nonlinear elliptic boundary
value problems}, J. Math. Anal. Appl.,  424 (2015), 598-612.

\bibitem{DJP} R. Dhanya, J. Giacomoni, S. Prashanth and K. Saoudi, {\it Global bifurcation and local
multiplicity results for elliptic equations with singular nonlinearity of super exponential
growth in $\mb R^2$}, Advances in differential equations, 17(3-4) (2012), 369-400.

\bibitem{DiMoOs} J.~I. D\'{\i}az, J.~M. Morel and L. Oswald, {\it An elliptic equation with singular
nonlinearity}, Commun. Partial Differential Equations, 12 (1987), 1333--1344.

\bibitem{DI} F. D\"{u}zg\"{u}n  and A. Iannizzotto, {\it hree nontrivial solutions for nonlinear fractional Laplacian equations}, Advances in Nonlinear Analysis, 2016, DOI:10.1515/anona-2016-0090

\bibitem{palatucci} G. Franzina and G. Palatucci, {\it Fractional p-eigenvalues}, Riv. math. Univ. Parma(N.S.), 5(2), 373--386.

\bibitem{GhRa} M. Ghergu and V. R\u{a}dulescu, {\sl Singular elliptic problems: bifurcation and asymptotic analysis}, Oxford University Press, 2008.

\bibitem{GhRa2} M. Ghergu and  V. R\u{a}dulescu, {\it Multiparameter bifurcation and asymptotics for the singular Lane-Emden-Fowler equation with a convection term}, Proceedings of the Royal Society of Edinburgh: Section A (Mathematics), 135 (2005), 61--84.

\bibitem{TJS} J. Giacomoni, T. Mukherjee and K. Sreenadh, {\it Positive solutions of fractional elliptic equation with critical and singular nonlinearity},   Adv. Nonlinear Anal.,  6 (3) (2017), 327-354.

\bibitem{TJS-para} J. Giacomoni, T. Mukherjee and K. Sreenadh, {\it Existence and stabilization results for a singular parabolic equation involving the fractional Laplacian}, to appear in Discrete and Continuous Dynamical Systems - Series S.


\bibitem{shiv3} J. Goddard II, E. Kyoung Lee, L. Sankar and R. Shivaji, {\it Existence results for classes of infinite
semipositone problems}, Bound. Value Probl., (2013), Article 97.

\bibitem{haitoh} Y. Haitao, {\it Multiplicity and asymptotic behavior of positive solutions for a singular semilinear elliptic problem}, J. Differential Equations, 189 (2003), 487-512.

\bibitem{HeMa} J. Hern\'andez and F.~J. Mancebo, {\sl Singular elliptic and parabolic equations}, Handbook of Differential Equations,   3 (2006), 317--400.

\bibitem{hcn1} N. Hirano, C. Saccon and N. Shioji, {\it Brezis-Nirenberg type theorems and multiplicity of positive solutions for a singular elliptic problem}, J. Differential Equations, 245 (2008), 1997-2037.

\bibitem{hcn2} N. Hirano, C. Saccon and N. Shioji, {\it Existence of multiple positive solutions for singular elliptic problems with concave and convex nonlinearities}, Advances in Differential Equations, 9(1-2) (2004), 197-220.

\bibitem{shiv2} Eun Kyoung Lee, R. Shivaji and J. Ye, {\it Positive solutions for infinite semipositone problems with falling zeros}, Nonlinear Analysis, 72 (2010), 4475-4479.

\bibitem{KLS} E. Ko, E. Kyoung Lee and R. Shivaji, {\it Multiplicity Results for Classes of
Infinite Positone Problems}, Z. Anal. Anwend., 30 (2011), 305-318.

\bibitem{radu2} G. Molica Bisci, V.D. Radulescu and R. Servadei, {\it Variational Methods for Nonlocal Fractional Problems}, Encyclopedia of Mathematics and its Applications, Cambridge: Cambridge University Press, Cambridge, (2016), DOI: 10.1017/CB09781316282397.

\bibitem{TS} T. Mukherjee and K. Sreenadh, {\it Critical growth fractional elliptic problem with singular nonlinearities}, Electronic Journal of differential equations, 54 (2016), 1-23.

\bibitem{shiv1} M. Ramaswamy, R. Shivaji and J. Ye, {\it Positive solutions for a class of infinite semipositone problems}, Differential and Integral Equations, 20(12) (2007), 1423-1433.

\bibitem{RoSe} X. Ros-Oton and J. Serra, {\it The Dirichlet problem for the fractional Laplacian: Regularity up to the boundary}, J. Math. Pures Appl., 101 (2014) 275--302.

\bibitem{silvestre} L. Silvestre,
{\it Regularity of the obstacle problem for a fractional power of the laplace operator}, {Comm. Pure Appl. Math.}, {60} (2007), 67--112.


    \end{thebibliography}
\end{document}